\documentclass[11pt,one,english]{amsart}
\usepackage{mathrsfs}
\usepackage{mathabx}
\usepackage{graphicx}
\usepackage{enumerate}
\usepackage{dsfont}
\usepackage[english]{babel}
\usepackage{xcolor}

\usepackage{tikz-cd}

\usepackage[utf8]{inputenc}
\usepackage{mathtools}
\usepackage{mathrsfs}
\usepackage{amsmath,amssymb,amsfonts, mathabx}
\usepackage[all,cmtip]{xy}
\usepackage{a4wide}
\usepackage{enumitem}
\usepackage{extarrow}
\usepackage{enumerate}
\usepackage{fontenc}
\usepackage[T1]{fontenc}
\usepackage{graphicx}
\usepackage[colorlinks=true, linkcolor=blue]{hyperref}
\usepackage{latexsym}
\usepackage{mathrsfs}
\usepackage{mathtext}
\usepackage{tikz}
\usepackage{textcomp}
\usepackage{upgreek}
\usepackage{xy}
\usetikzlibrary{arrows}
\usetikzlibrary{matrix}
\usetikzlibrary{shapes}
\usetikzlibrary{snakes}
\usetikzlibrary{matrix}

\usepackage[
backend=biber,
style=alphabetic-verb,
]{biblatex}

\AtEveryBibitem{\clearlist{language}}

\DeclareMathOperator{\Coker}{\textup{Coker}}

\DeclareMathOperator{\Der}{\textup{Der}}

\DeclareMathOperator{\End}{\textup{End}}

\DeclareMathOperator{\Hl}{\textup{H}}

\DeclareMathOperator{\Imm}{\textup{Im}}

\DeclareMathOperator{\lf}{\textup{LF}}

\DeclareMathOperator*{\colim}{\textup{colim}}
\usepackage[disable]{todonotes}

\theoremstyle{plain}
\newtheorem{thm}{Theorem}[section]
\newtheorem*{thm*}{Theorem}
\newtheorem{lem}[thm]{Lemma}
\newtheorem{cor}[thm]{Corollary}
\newtheorem{prop}[thm]{Proposition}

\theoremstyle{remark}
\newtheorem{rmk}[thm]{Remark}

\newtheorem*{rmk*}{Remark}

\theoremstyle{definition}
\newtheorem{defn}{Definition}[section] 

\newtheorem*{const*}{Construction}
\newtheorem{conv}[defn]{Setting}

\numberwithin{thm}{section}
\numberwithin{equation}{section}

\DeclareMathOperator{\Diff}{Diff}
\DeclareMathOperator{\Hom}{Hom}
\DeclareMathOperator{\im}{Im}

\newcommand{\Spec}{{\rm Spec \,}}
\newcommand{\Crys}{{\rm Crys}}
\newcommand{\crys}{{\rm crys}}
\newcommand{\QNCf}{{\rm QNCf}}
\newcommand{\Ker}{{\rm Ker}}
\newcommand{\Zar}{{\rm Zar}}
\newcommand{\Conn}{{\rm Conn}}
\newcommand{\Icrys}{I_{\textrm{crys}}}
\newcommand{\sA}{{\mathcal A}}
\newcommand{\sC}{{\mathcal C}}
\newcommand{\sD}{{\mathcal D}}
\newcommand{\der}{{ {\sD}er}}
\newcommand{\sE}{{\mathcal E}}
\newcommand{\sF}{{\mathcal F}}
\newcommand{\sH}{{\mathcal H}}
\newcommand{\sHom}{{\sH om}}
\newcommand{\sO}{{\mathcal O}}
\newcommand{\rr}{{\mathbf{R}}}
\newcommand{\dl}{{\mathbf{L}}}
\newcommand{\Mod}{\text{\sf Mod}}
\newcommand{\Vect}{\text{\sf Vect}}

\global\long\def\E{\mathcal{E}}

\global\long\def\N{\mathbb{N}}

\global\long\def\Z{\mathbb{Z}}

\global\long\def\arr{\longrightarrow}

\global\long\def\arrdi#1{\xlongrightarrow{#1}}

\global\long\def\comma{,\ }

\global\long\def\id{\textup{id}}

\global\long\def\odi#1{\mathcal{O}_{#1}}

\global\long\def\stG{\mathcal{G}}

\global\long\def\then{\ \Longrightarrow\ }

\newcommand{\lS}{{\mathbf S}}

\newcommand{\lT}{{\mathbf T}}
\newcommand{\lX}{{\mathbf X}}

\newcommand{\lD}{{\mathbf D}}
\newcommand{\lW}{{\mathbf W}}
\newcommand{\lV}{{\mathbf V}}
\newcommand{\lU}{{\mathbf U}}
\newcommand{\lZ}{{\mathbf Z}}
\newcommand{\lP}{{\mathbf P}}
\renewcommand{\lf}{{\mathbf f}}

\renewcommand{\to}{\arr}

\theoremstyle{plain}
\newtheorem{thmI}{Theorem}
\newtheorem{thmII}{Theorem}
\newtheorem{thmIII}{Theorem}
\newtheorem{thmIV}{Theorem}

\begin{document}
\title[A crystalline incarnation of Berthelot's conjecture]{A crystalline incarnation of Berthelot's conjecture and Künneth formula for isocrystals}
\author{Valentina Di Proietto, Fabio Tonini, Lei Zhang}

\address{ 
Valentina Di Proietto\\
    University of Exeter\\
    College of Engineering, Mathematics and Physical Sciences\\
Streatham Campus\\
    Exeter, EX4 4RN\\
United Kingdom }
\email{v.di-proietto@exeter.ac.uk}

\address{
    Università degli Studi di Firenze\\
    Dipartimento di Matematica e Informatica Ulisse Dini\\
    Viale Morgagni 67/a\\ Firenze, 50134 Italy }
\email{fabio.tonini@unifi.it}

 \address{
    The Chinese University of Hong Kong\\
    Department of Mathematics\\    
    Shatin, New Territories\\ Hong Kong }
\email{lzhang@math.cuhk.edu.hk} 

\date{\today}

\thanks{This work was supported by the European Research Council (ERC) Advanced Grant 0419744101 and the Einstein Foundation. Part of the revision of this work has been done while the first author was guest of the IMPAN: her stay was supported by the grant 346300 for IMPAN from the Simons Foundation and the matching 2015--2019 Polish MNiSW fund. The second author was
supported by GNSAGA of INdAM}

\makeatletter 
\providecommand\@dotsep{5} 
\makeatother 

\begin{abstract}
 Berthelot's conjecture predicts that under a proper and smooth
 morphism of schemes in characteristic $p$, the higher direct
 images of an overconvergent $F$-isocrystal are overconvergent
 $F$-isocrystals. In this paper we prove that this is true for
 crystals up to isogeny. As an application we prove the
 K\"unneth formula for the crystalline fundamental group scheme.
\end{abstract}
\maketitle
\section*{Introduction}
One of the expectations for a good cohomology theory for schemes is that there exists a pushforward functor $f_*$ associated to a proper and smooth morphism $f\colon X\rightarrow S$ such that $\rr^q f_*$ (for $q\geq 0$) sends a coefficient for the cohomology on $X$ to a coefficient for the cohomology on $S$. This expectation is reality in various contexts.

 Let $k$ be a field of characteristic $0$, $f\colon X\rightarrow S$ be a proper and smooth morphism between two $k$-varieties, and let $\mathcal{E}$ be a module with integrable connection on $X$; then the relative de Rham cohomology $\rr^qf_*(\mathcal{E})$ comes endowed with an integrable connection, the Gauss--Manin connection (see for example \cite{Katz70}, \cite{Har75}), so that it is indeed a coefficient for the cohomology on $S$. 

When $k$ is a field of characteristic $p>0$, $f\colon X\rightarrow S$ is a proper and smooth morphism between two $k$-varieties, and $\mathcal{E}$ an $\ell$-adic lisse sheaf ($\ell\neq p$), then $\rr^qf_*(\mathcal{E})$ is an $\ell$-adic lisse sheaf (\cite{SGA41/2}).  

As for the case $\ell=p$, the expectation for an overconvergent $F$-isocrystal $\mathcal{E}$ is known as Berthelot's conjecture (\cite[(4.3)]{Ber86}, \cite{Tsu03}). The conjecture is still open, but several results have been obtained in the last years (\cite{Tsu03}, \cite{Shi07I}, \cite{Shi07II}, \cite{Shi07III}  \cite{Car15}, \cite{Ete12}, \dots). For a survey about this conjecture see \cite{Laz16}. 

As remarked by Ladza in \cite{Laz16}, Berthelot's conjecture can have many incarnations, depending on what kind of coefficients and pushforward one considers. In this paper we deal with a crystalline incarnation of Berthelot's conjecture, working with the category of crystals up to isogeny on the crystalline site.

Let $k$ be a perfect field of characteristic $p>0$, let $W$ be
the ring of Witt vectors of $k$ and let $K$ be the fraction
field of $W$. Set $\lW\coloneq \Spec W$. For a $k$-scheme $X$, Berthelot defined the
crystalline site $(X/\lW)_{\crys}$ and the structure sheaf
$\mathcal{O}_{X/\lW}$. He considered also the category of
crystals of finite presentation, denoted by $\Crys(X/\lW)$,
defined as the category of certain sheaves of $
\mathcal{O}_{X/\lW}$-modules on $(X/\lW)_{\crys}$ which verify a
rigidity condition. The category of isocrystals
$I_{\crys}(X/\lW)$ is the category  $\Crys(X/\lW)$ up to
isogeny, \emph{i.e.} the category whose objects are exactly those in
$\Crys(X/\lW)$ and whose morphisms are obtained inverting the multiplication by $p$. Thus we have a natural functor
$$\Crys(X/\lW)\arr I_{\crys}(X/\lW)$$
which is the identity on objects. To distinguish among objects
in $\Crys(X/\lW)$ and in $I_{\crys}(X/\lW)$ we denote by
$K\otimes E$ the image of $E\in\Crys(X/\lW)$ under the above
functor, and we say that $E$ is a lattice for the isocrystal $\mathcal{E}$ if $K\otimes E\cong \mathcal{E}$.

Given a proper and smooth morphism of $k$-schemes $f\colon X\rightarrow S$ and a crystal $E$ on the crystalline site $(X/\lW)_{\crys}$, there is a morphism of ringed topoi 
\[f_{\rm{crys}*}:((X/\lW)^\sim_{\rm{crys}}, \mathcal{O}_{X/\lW})\longrightarrow ((S/\lW)^\sim_{\rm{crys}}, \mathcal{O}_{S/\lW})\]
and its derived version $\rr f_{\rm{crys}*}$. 
By functoriality the functors $f_{\rm{crys}*}$ and $\rr
f_{\rm{crys}*}$ induce corresponding functors in the isogeny
categories, so if $\mathcal{E}$ is an isocrystal 
in $I_{\rm{crys}}(X/\lW)$,  then  for all $q\geq0$, we get   an
object $\rr^qf_{\mathrm{crys}*}(\mathcal{E})$ in the isogeny category of $\mathcal{O}_{S/\lW}$-modules. The main result of the paper is that, if $S$ is smooth, $\rr^qf_{\mathrm{crys}*}(\mathcal{E})$ has a richer structure, indeed it is an isocrystal, \emph{i.e.} an object of $I_{\rm{crys}}(S/\lW)$.

\begin{thmI}\label{main theorem}
Let $f\colon X\rightarrow S$ be a smooth and proper morphism of smooth quasi-compact $k$-schemes and let $\mathcal{E}$ be an isocrystal 
in $I_{\rm{crys}}(X/\lW)$. Then, for all $q\geq0$,  $\rr^qf_{\mathrm{crys}*}(\mathcal{E})$ is an isocrystal in $I_{\rm{crys}}(S/\lW)$.
\end{thmI}

The above theorem generalises a result of Morrow, which proved the above theorem for the trivial isocrystal (\cite{Mor18}). Our proof follows the lines of his proof: we explain here the main ideas. 

First, using Zariski descent, one reduces to the case in which
$S=\Spec A$ is affine; now $A$ can be lifted to a $p$-adically
complete flat $W$-algebra $\mathcal{A}$, such that
$\sA_n=\mathcal{A}/p^n\sA$ is a smooth $W_n\coloneq W/p^nW$ algebra for
all $n\geq 1$. Set $\lW_n\coloneq\Spec W_n$. Since $X$ is smooth
over $k$, there exists a $p$-torsion free crystal $E$ on $X$
which is a lattice for $\mathcal{E}$, then one has a
Gauss--Manin crystal at one's disposal. Indeed, given a
$p$-torsion free crystal $E$ on $\Crys(X/\lW_n)_{\crys}$, one
can construct a natural HPD-stratification on the finitely
generated  $\sA$-module $\varprojlim_n(\rr^q
f_{\rm{crys}*}(E))_{\Spec\sA_n}$ over $W$. Using the fact that
$\sA_n$ is $W_n$-smooth for all $n\in\N^+$, the HPD-stratification on $\varprojlim_n(\rr^q f_{\rm{crys}*}(E))_{\Spec\sA_n}$ is equivalent to  a crystal $E_{X/\sA}^q$ on $(S/\lW)_{\crys}$ -- the Gauss--Manin crystal. Moreover, there is a natural map $$E_{X/\sA}^q\longrightarrow \rr^q f_{\rm{crys}*}(E)$$ of sheaves on $(S/\lW)_{\crys}$ which turns out to be an isomorphism after 
inverting $p$. This shows that $\rr^q
f_{\rm{crys}*}(E)\otimes K$ (see Definition \ref{higher direct image up to
isogeny}) is in $I_{\rm{crys}}(S/\lW)$.

A key ingredient of the above proof is the Berthelot's base change theorem for crystalline cohomology \cite[Theorem 7.8]{BO78} which only holds for flat crystals.
In Morrow's paper  the trivial isocrystal $K\otimes
\mathcal{O}_{X/\lW}$  admits a lattice, e.g.  $\sO_{X/\lW}$, which is
flat, that is, $-\otimes\sO_{X/\lW}$ is exact in the ringed
topos $((X/\lW)^\sim_\crys, \sO_{X/\lW})$.  But in general the existence of a flat lattice is not known (see for example \cite{ES15}). This becomes a central theme of  this paper: in \S\ref{section:basechange} we develop a crystalline base change theory for crystals that may not be flat; 
instead of requiring that the base change map is an isomorphism we require that it is an isomorphism after inverting $p$. The proof follows closely the original proof of Berthelot's base change theorem, namely it uses cohomological descent to reduce the problem to the affine case and then work with the quasi-nilpotent connections and the corresponding de Rham complex. But the argument from there on has to be changed due to the lack of the flatness condition. We have to use a spectral sequence to find a uniformly large $N$ so that $p^N$ kills both the kernel and the cokernel of the base change map. Shiho also studied in \cite{Shi07I} isocrystals which do not necessarily admit  flat lattices, but his results do not fit our situation.

We prove several variants of base change isomorphisms (see Theorem
\ref{Toninibasechange}, Theorem \ref{Toninibasechange Gamma} and Theorem \ref{basechange for crystalpush}). Here we mention the following.

\begin{thmII}\label{underived base change} Consider a cartesian diagram
   \[
  \begin{tikzpicture}[xscale=1.5,yscale=-1.2]
    \node (A0_0) at (0, 0) {$X'$};
    \node (A0_1) at (1, 0) {$X$};
    \node (A1_0) at (0, 1) {$S'$};
    \node (A1_1) at (1, 1) {$S$};
    \path (A0_0) edge [->]node [auto] {$\scriptstyle{h}$} (A0_1);
    \path (A0_0) edge [->]node [auto] {$\scriptstyle{f'}$} (A1_0);
    \path (A0_1) edge [->]node [auto] {$\scriptstyle{f}$} (A1_1);
    \path (A1_0) edge [->]node [auto] {$\scriptstyle{v}$} (A1_1);
  \end{tikzpicture}
  \]
of quasi-compact $k$-schemes with $f$ smooth and proper. Let
 $E\in\Crys(X/\lW)$ and assume $S$ is smooth over $k$. Then for all $n\in \N$ the canonical map
 \[
 v^*_\crys\rr^nf_{\crys *}(E)\arr \rr^nf'_{\crys *}(h_\crys^*E)
 \]
 is an isomorphism of isocrystals in $I_{\rm{crys}}(S'/\lW)$.
\end{thmII}

A recent result proven by Xu (\cite{Xu2019}) deals with a convergent incarnation of Berthelot's conjecture: he proves that the derived pushforward functor preserves convergent isocrystals, in the context of the convergent topos defined by Ogus \cite{Ogus84}. Let $f\colon X\rightarrow S$ be a proper and smooth map as above; Xu considers a convergent isocrystal $\mathcal{E}\in I_{\rm{conv}}(X/\lW)$, together with $\rr^qf_{\rm{conv}*}(\mathcal{E})$; he uses Shiho’s base change \cite[Theorem 1.19]{Shi07I} to show that $\rr^qf_{\rm{conv}*}(\mathcal{E})$ is a $p$-adically convergent isocrystal. Then he develops a strong version of Frobenius descent which allows him to prove that $\rr^qf_{\rm{conv}*}(\mathcal{E})$ is indeed a convergent isocrystal on $S$ using Dwork’s trick. He then proceeds to remove the smoothness hypothesis for the base $S$. It would be interesting to know if even in our setting one can remove the smoothness hypothesis. In any case, when $S$ is smooth over $k$,  the category of convergent isocrystals is a full subcategory of the category of isocrystals \cite[Theorem 0.7.2]{Ogus84}: there is a fully faithful functor
$\iota: I_{\rm{conv}}(S/\lW)\rightarrow I_{\rm{\crys}}(S/\lW)$
(and likewise for $X/\lW$). Our result and Xu's result are
independent, in the sense that none of the two implies the
other. On the other hand they are compatible in the sense that $\iota(\rr^q f_{\rm{conv}*}
\mathcal{E}) \cong \rr^qf_{\crys*}(\iota(\mathcal{E}))$ (see Remark \ref{Xu comparison} and the discussion at the end of \cite[Section 1.9]{Xu2019}).

We remark that if $X$ is a smooth, quasi-compact and connected $k$-scheme, then the category $I_{\rm{crys}}(X/\lW)$ is a Tannakian category, hence when $X$ has a $k$-rational point $x$, one can define the crystalline fundamental group $\pi_1^{\rm{crys}}(X/\lW, x)$\footnote{While the authors were revising this paper, a preprint with a new approach to crystals appeared \cite{Dri18}.}. This group scheme has recently been studied deeply: it has been conjectured by de~Jong that for a connected projective variety over an algebraically closed field in characteristic $p>0$ with trivial étale fundamental group, there are no non-constant isocrystals. The conjecture is still open but several results have been obtained (\cite{Kats18}, \cite{ES18}, \cite{ES16}, \cite{Shi14}). 
Moreover, we also remark that the pro-unipotent completion of $\pi_1^{\rm{crys}}(X/\lW, x)$ is considered to be the crystalline realisation of the motivic fundamental group and it has been studied by Shiho in the more general context of log geometry (\cite{Shi00}, \cite{Shi02}).

As a consequence of our main result we obtain the Künneth formula for the crystalline fundamental group. 

\begin{thmIII}\label{Kunneth}
 Let $k$ be a perfect field of characteristic $p>0$, let $X$ and $Y$ be smooth connected $k$-schemes with $Y$ proper and suppose that $x\in X(k)$, $y\in Y(k)$ are two rational points.
 Then the canonical morphism between the crystalline fundamental groups
 $$\pi_1^{\rm{crys}}(X\times_kY/\lW,(x,y))\arr \pi_1^{\rm{crys}}(X/\lW, x) \times_K \pi_1^{\rm{crys}}(Y/\lW, y)$$
 is an isomorphism.
\end{thmIII}
By the Eckman--Hilton argument we also get the following. 
\begin{thmIV}\label{Abelian}
 Let $A$ be an abelian variety over a perfect field $k$ of positive characteristic. Then $\pi_1^{\rm{crys}}(A/\lW, 0)$ is an abelian group scheme.
\end{thmIV}
Analogous results for other fundamental groups have been obtained by Battiston
\cite{Bat16} and D'Addezio \cite{DAd21}. 

The Künneth formula, as in the étale case, is a consequence of the homotopy
exact sequence for the crystalline fundamental group, but our argument does not
use the homotopy exact sequence. It is an open problem to show the existence of
a homotopy exact sequence for the crystalline fundamental group, which has been
shown is several other contexts recently (\cite{Zhang2014b}, \cite{dS2015}, \cite{LP17}, \cite{DiPShi18}, \dots).

The content of each section is as follows. In \S\ref{section:tannakian} we define the crystalline fundamental group; to do so we prove that the category of isocrystals on a smooth, quasi-compact and connected $k$-scheme is Tannakian. In \S\ref{section:basechange} we prove several generalisations of the base change for crystalline cohomology. We consider a PD-scheme $\lS=(\lS, I, \gamma)$ over $\lW$, requiring that $p
\in I$, we let $S=V(I)$, and we consider an $S$-scheme $X$. We
denote by $g\colon X\arr \lS$ the structure map, and by
$g_{X/\lS}$ the morphism of topoi $g\circ u_{X/\lS}\colon
(X/\lS)^{\sim}_{\rm{crys}}\rightarrow \lS^{\sim}_{\rm{Zar}}$. In \S\ref{subsection:killed_by_p} we prove the generalised base change theorem for $g_{X/\lS}$ when $p$ is nilpotent in $\mathcal{O}_{\lS}$; this includes, as a special case, the classical Berthelot's base change theorem for crystalline cohomology. In \S\ref{subsection:affine} we consider the case in which $\lS$ is affine. In this case we consider the functor $\varprojlim_n\Gamma \circ g_{X/(\lS/p^n)*}$; we prove a base change theorem for this functor. In the last part of section \S\ref{section:basechange} we consider a proper and smooth morphism of smooth $k$-schemes $f\colon X\rightarrow S$ as above, and we prove a base change theorem for the functor $f_{\crys*}$. In section \S\ref{section:pushforward}, we get our main result about Berthelot's conjecture for isocrystals. In \S\ref{section:kunneth} we prove the Künneth formula for the crystalline fundamental group.

\section*{Acknowledgements}
This work started thanks to a question of Tomoyuki Abe: he asked us whether a Künneth formula for the crystalline fundamental group would have been true. We want to thank him sincerely also for several discussions related to this paper and for his encouragement. We had a very useful conversation with Atsushi Shiho in Berlin about this work: it is a great pleasure to thank him deeply. We are very thankful to Luc Illusie for the interest he showed in our work when we presented these results during a conference at Technical University Münich. The last corollary was suggested by Marco D'Addezio: we want to thank him also for his interest in our work. The authors would like to acknowledge also a useful conversation about derived categories they had with Shane Kelly. We thank Matthew Morrow for his answer to a question about his paper.

Finally, we would like to thank the anonymous referee who did an impressive job of revision. 
Her/his comments greatly improved the shape of our paper, both mathematically and
exposition-wise.

\section*{Notation}

An ideal $I$ in a ring $A$ is called nil if all elements of $I$ are nilpotent (this is called a locally nilpotent ideal in \cite{stacks-project}). We will often use that smooth affine maps have the lifting property for nil ideals (see \cite[\href{https://stacks.math.columbia.edu/tag/07K4}{Tag 07K4}]{stacks-project}).
\section{Tannakian categories of connections and crystalline fundamental group}\label{section:tannakian}
The goal of this section is to define the crystalline fundamental group. Concretely this means introducing the category of isocrystals and proving that it is a Tannakian category. This is done in essentially four steps.
\begin{enumerate}
\item Reduce the problem to the affine case and compare isocrystals with topologically quasi-nilpotent connections.
\item Interpret topologically quasi-nilpotent connections as a particular case of connections with respect to a quotient of the sheaf of algebraic differentials.
\item Show that those connections correspond to differential modules for an associated differential ring.
\item Study differential rings and differential modules following \cite{Ked12}.
    \end{enumerate}
This program is done in the reverse order, so that definitions come first.

\subsection{Differential rings}
We start by introducing some general definitions as in \cite{Ked12}.
\begin{defn}\label{differential modules}
 A differential ring is a pair $(A,\Delta_A)$ where $A$ is a ring and $\Delta_A$ is a Lie algebra
 together with an $A$-module structure and a Lie algebra homomorphism $\iota\colon \Delta_A\arr \Der(A/\Z)=\Hom(\Omega_{A/\Z},A)$ which is $A$-linear. We moreover ask that the following         property holds:
 \begin{equation}\label{eq:differential ring condition}
  [D_1,aD_2]=a[D_1,D_2]+D_1(a)D_2 \text{ for all }a\in A\comma
  D_1,D_2\in \Delta_A.
 \end{equation}
 Notice that the above equation is automatic if $\iota$ is injective. If $X=\Spec A$ we sometimes write $(X,\Delta_A)$ instead of $(A,\Delta_A)$.
 
 If $B$ is a ring, a differential $B$-algebra is a differential ring $(A,\Delta_A)$ such that $A$ is a $B$-algebra and the map $\iota\colon \Delta_A\arr \Der(A/\Z)$ has image in $\Der(A/B)=\Hom(\Omega_{A/B},A)$.
 
 A differential $(A,\Delta_A)$-module (or simply differential $A$-module when $\Delta_A$ is clear from the context) is a pair $(M,\nabla)$ where $M$ is an $A$-module and
 \[
 \nabla\colon \Delta_A \arr \End_\Z(M)
 \]
 is a morphism of Lie algebras which is $A$-linear and satisfies the Leibniz rule, \emph{i.e.}
 \begin{equation}\label{eq:Leibniz rule}
  D(am)=D(a)m+aD(m) \text{ for all } D\in \Delta_A\comma a\in
  A\comma m\in M.
 \end{equation}
 Above and in what follows we write $D(m)$ instead of $\nabla(D)(m)$. 
 
 
 We denote by $\Diff(A,\Delta_A)$ or simply $\Diff(A)$ the category of differential $A$-modules which are of finite presentation (as $A$-modules).

\end{defn}

\begin{rmk}
    Let $(A,\Delta_A)$ be a differential ring and $(E,\nabla_E)$ and $(F,\nabla_F)$ be differential $A$-modules. Their tensor product is given by the $A$-module $E\otimes_{A} F$ and the map
      \[
  \begin{tikzpicture}[xscale=5.4,yscale=-0.6]
    \node (A0_0) at (0, 0) {$\Delta_A$};
    \node (A0_1) at (1, 0) {$\End_\Z(E\otimes_{A} F)$};
    \node (A1_0) at (0, 1) {$D$};
    \node (A1_1) at (1, 1) {$(e\otimes f\mapsto D(e)\otimes
        f+e\otimes D(f)).$};
    \path (A0_0) edge [->]node [auto] {$\scriptstyle{\nabla_{E\otimes_{A} F}}$} (A0_1);
    \path (A1_0) edge [|->,gray]node [auto] {$\scriptstyle{}$} (A1_1);
  \end{tikzpicture}
  \]
Their Hom is instead given by 
the $A$-module $\Hom_{A}(E,F)$ and the map
  \[
  \begin{tikzpicture}[xscale=5.4,yscale=-0.6]
    \node (A0_0) at (0, 0) {$\Delta_A$};
    \node (A0_1) at (1, 0) {$\End_\Z(\Hom_A(E,F))$};
    \node (A1_0) at (0, 1) {$D$};
    \node (A1_1) at (1, 1) {$(\phi\mapsto \nabla_F(D)\circ \phi
        -\phi\circ \nabla_E(D)).$};
    \path (A0_0) edge [->]node [auto] {$\scriptstyle{\nabla_{\Hom_A(E,F)}}$} (A0_1);
    \path (A1_0) edge [|->,gray]node [auto] {$\scriptstyle{}$} (A1_1);
  \end{tikzpicture}
  \]
See also \cite[Def.~1.1.3]{Ked12}. It is easy to see that the category of differential $A$-modules is symmetric monoidal with unit $(A,\nabla)$ where $\nabla\colon \Delta_A\arr \Der(A/\Z)\subseteq \End_\Z(A)$ is the canonical map $\iota$. Moreover the Hom just defined is an internal Hom in the category of differential $A$-modules, that is if $(G,\nabla_G)$ is another differential $A$-module then the canonical isomorphism
\[
\Hom(E\otimes F,G)\arr \Hom(E,\Hom(F,G))
\]
is a map of differential $A$-modules and preserves the subsets of morphisms of differential $A$-modules.

If $E\arrdi \phi F$ is a map of differential $A$-modules then kernel and cokernels are naturally differential $A$-modules.
\end{rmk}

From the discussion above we can conclude that

\begin{prop}\label{first prop of differential modules}
 If $(A,\Delta_A)$ is a differential ring then the category of differential $A$-modules is symmetric monoidal, abelian and has internal homomorphisms. The same is true for $\Diff(A)$ if $A$ is Noetherian.
\end{prop}

\begin{defn}

    Let $R$ be a ring. Let $\mathscr{C}$ be an $R$-linear category and let $R'$ be an $R$-algebra. We denote by $\mathscr{C}\otimes_RR'$
    the category whose objects are exactly those of $\mathscr{C}$ and whose morphisms are given by
    $$\Hom_{\mathscr{C}\otimes_R R'}(M,N)\coloneqq \Hom_{\mathscr{C}}(M,N)\otimes_RR'$$
    for any $M,N\in\mathscr{C}$. There is a natural functor $F\colon\mathscr{C}\arr\mathscr{C}\otimes_RR'$ which is
    the identity on objects and which is the natural base extension on morphisms. For any object
    $M\in \mathscr{C}$, in order to emphasize that $F(M)$ is in $\mathscr{C}\otimes_RR'$, we write $M\otimes_RR'$ for $F(M)$.
    
    If $\mathscr{C}$ is symmetric monoidal then also $\mathscr{C}\otimes_R R'$ is symmetric monoidal in a natural way.

\end{defn}

\begin{lem}\label{localization abelian} Let $R$ be a ring. 
    Let $\mathscr{C}$ be an $R$-linear abelian category, and let $S$ be a multiplicative subset of $R$. Then $\mathscr{C}\otimes_R S^{-1}R$ is also abelian and the natural functor
    $F\colon \mathscr{C}\arr \mathscr{C}\otimes_RS^{-1}R$ is exact. Moreover if $\mathscr{C}\arr\mathscr{D}$ is an $R$-linear exact functor to an $S^{-1}R$-linear category, then the induced functor $\mathscr{C}\otimes S^{-1}R\arr \mathscr{D}$ is also exact.
    
    If $\mathscr{C}$ is symmetric monoidal (with internal
    homomorphisms) then $F\colon \mathscr{C}\arr
    \mathscr{C}\otimes_R S^{-1}R$ is a tensor functor (and preserves internal homomorphisms).
\end{lem}
\begin{proof}
Set $R'=S^{-1}R$.
    Since up to isomorphisms every morphism in $\mathscr{C}\otimes_RR'$ comes from $\mathscr{C}$, in order to show the exactness of $F$ it is enough to show that $F$ preserves kernel and cokernel. Let's look at kernel for example. Let $f\colon A\to B$ be a morphism in $\mathscr{C}$. Then $\Ker(f)$ is the 
    object in $\mathscr{C}$ which represents the functor that sends any $T\in \mathscr{C}$ to 
    $$\Ker(\Hom_\mathscr{C}(T,A)\arr\Hom_\mathscr{C}(T,B)).$$ By the flatness of $R\to R'$ we have the exact sequence
              \[
  \begin{tikzpicture}[xscale=4.6,yscale=-1]
    \node (A0_0) at (0, 0) {$0$};
    \node (A0_1) at (0.8, 0) {$\Hom_\mathscr{C}(T,\Ker(f))\otimes R'$};
    \node (A0_2) at (2, 0) {$\Hom_\mathscr{C}(T,A)\otimes R'$};
    \node (A0_3) at (3, 0) {$\Hom_\mathscr{C}(T,B)\otimes R'$};
    \node (A1_0) at (0, 1) {$0$};
    \node (A1_1) at (0.8, 1) {$\Hom(T\otimes R',\Ker(f)\otimes R')$};
    \node (A1_2) at (2, 1) {$\Hom(T\otimes R',A\otimes R')$};
    \node (A1_3) at (3, 1) {$\Hom(T\otimes R',B\otimes R')$};
    \path (A0_1) edge [-,double distance=1.5pt]node [auto] {$\scriptstyle{}$} (A1_1);
    \path (A0_0) edge [->]node [auto] {$\scriptstyle{}$} (A0_1);
    \path (A0_1) edge [->]node [auto] {$\scriptstyle{}$} (A0_2);
    \path (A1_0) edge [->]node [auto] {$\scriptstyle{}$} (A1_1);
    \path (A0_3) edge [-,double distance=1.5pt]node [auto] {$\scriptstyle{}$} (A1_3);
    \path (A1_1) edge [->]node [auto] {$\scriptstyle{}$} (A1_2);
    \path (A0_2) edge [-,double distance=1.5pt]node [auto] {$\scriptstyle{}$} (A1_2);
    \path (A1_2) edge [->]node [auto] {$\scriptstyle{}$} (A1_3);
    \path (A0_2) edge [->]node [auto] {$\scriptstyle{}$} (A0_3);
  \end{tikzpicture}
  \]
for each $T\otimes_RR'\in\sC\otimes_RR'$. Thus $\Ker(f)\otimes_RR'$ represents the kernel of $f\otimes_RR'$.

Now consider an exact linear functor $G\colon\mathscr{C}\arr
\mathscr{D}$ as in the statement and call $G'\colon
\mathscr{C}\otimes R'\arr\mathscr{D}$ the induced functor. Let
$A'_{\bullet}$ be a bounded exact complex in $\mathscr{C}\otimes
R'$. In order to show that $G'(A'_{\bullet})$ is exact we can
multiply each degree map by elements of $S$. In particular we
can assume that all those maps are defined in $\mathscr{C}$ and,
multiplying again by elements of $S$, that they define a complex $A_{\bullet}$ such that $F(A_{\bullet})=A'_{\bullet}$. Using the exactness of $F$ and $G$ we have
\[
F(\Hl^i(A_{\bullet}))\simeq \Hl^i(A'_{\bullet})=0\then
0=G'F(\Hl^i(A_{\bullet}))=G(\Hl^i(A_{\bullet}))\simeq
\Hl^i(G(A_{\bullet}))\simeq \Hl^i(G'(A'_{\bullet})).
\]

The last statement follows from a direct check.
 \end{proof}
 
 \begin{rmk}\label{localize a differential module}
  Let $(A,\Delta_A)$ be a differential ring and $S$ be a multiplicative subset of $A$. Then $(S^{-1}A,S^{-1}\Delta_A)$ has a natural structure of differential ring. Moreover if $(M,\nabla)$ is a differential $A$-module then $S^{-1}M$ is a differential $S^{-1}A$-module in a natural way.
  
  The condition (\ref{eq:differential ring condition}) in Definition \ref{differential modules} forces the definition of the bracket in $S^{-1}\Delta_A$ as well as in $\Der(S^{-1}A)=\Hom(S^{-1}\Omega_{A/\mathbb{Z}},S^{-1}A)$. 
  
Also, the Leibniz rule (\ref{eq:Leibniz rule})  in Definition \ref{differential modules} forces the definition of the map $\nabla\colon S^{-1}\Delta_A\to \End(S^{-1}M)$: this is the unique $S^{-1}A$-linear map such that
  \[
  D(m/s)=D(m)/s - D(s)m/s^2.
  \]
  Indeed everything is well-defined (\cite[Rem.~1.1.5]{Ked12}).
 \end{rmk}

 \begin{lem}\label{localization}
 Let $R$ be a ring and let $(A,\Delta_A)$ be a differential $R$-algebra such that $\Delta_A$ is a finitely generated $A$-module and let $S$ be a multiplicative subset of $R$. Then $\Diff(A)$ is an $R$-linear category and the functor
   \[
  \Diff(A)\otimes_R S^{-1}R \arr \Diff(S^{-1}A)
  \]
  is a fully faithful tensor functor.
  If $A$ is Noetherian then the above functor is also exact and preserves internal homomorphisms.
 \end{lem}
\begin{proof}
Set $R'=S^{-1}R$ and $A'=S^{-1}A$.
The fact that the functor is a tensor functor follows from construction. For the full faithfulness, given two differential $A$-modules $(M,\nabla_M)$ and $(N,\nabla_N)$ we want to show that the natural map 
   \[
    \phi\colon \Hom_{\Diff(A)}(M,N)\otimes_RR'\longrightarrow
\Hom_{\Diff(A')}(M',N')\]
    is an isomorphism, where $M'\coloneqq M\otimes_RR'$ and
    $N'\coloneqq N\otimes_RR'$ are thought of as differential $A'$-modules.  The canonical map 
    $$\Hom_A(M,N)\otimes_RR'\longrightarrow\Hom_{A'}(M',N')$$
    is an isomorphism. Thus we have to show that
     if $f\in\Hom_{A'}(M',N')$ is a morphism which is compatible with the $\nabla_{*'}$, then it comes from
     \[\Hom_{\Diff(A)}(M,N)\otimes_RR'\subseteq
     \Hom_{A}(M,N)\otimes_RR'=\Hom_{A'}(M',N').\]
     Replacing $f$ by $sf$ for some $s\in S$ we may assume that $f$ comes from $\Hom_A(M,N)$ and we will still use $f$ to denote the lift of $f$ in $\Hom_A(M,N)$. 
     We must show that there exists $s\in S$ such that $sf$ preserves the $\nabla_*$. For $D\in \Delta_A$ and $m\in M$ set 
     \[
     g(D,m)=f(D(m))-D(f(m))\in N.
     \]
     Since $\nabla_N(D)$ is $R$-linear we look for an $s\in S$ such that $sg(D,m)=0$ for all $D$ and $m$. By hypothesis $g(D,m)=0$ in $N'=S^{-1}N$. Thus it is enough to notice that, by the Leibniz rule, $g(D,m)$ is a linear combination of the values of $g$ on generators of $\Delta_A$ and $M$, which are finitely many.
     
     Now assume that $A$ is Noetherian. Then the functor in the statement
     preserves internal homomorphisms because of how they are constructed and
     because all modules considered are finitely generated. The exactness
     follows from Lemma \ref{localization abelian}.
\end{proof}

\begin{defn}\cite[Def.~1.2.1]{Ked12}
 A differential ring $(A,\Delta_A)$ is called locally simple if for all prime ideals $P$ the differential local ring $A_P$ is simple, \emph{i.e.} $A_P$ contains no proper non zero ideals stable under the action of $(\Delta_A)_P$.  
\end{defn}

\begin{prop}\textup{\cite[Prop.~1.2.6]{Ked12}}\label{locally simple implies locally free}
 Let $A$ be a locally simple differential ring. If $(E,\nabla)$ is a differential $A$-module of finite presentation then $E$ is locally free as an $A$-module.
\end{prop}

\begin{thm}\label{locally simple implies Tannakian}
 Let $(A,\Delta_A)$ be a Noetherian locally simple differential ring such that $\Spec(A)$ is connected. Then $\Diff(A)$ is a Tannakian category over some subfield $L\subseteq A$. 
 Let $k$ be a field, let $A$ be differential $k$-algebra and $x\colon \Spec k\arr \Spec A$ be a rational point, then $\Diff(A)$ with the fiber functor obtained via $x^*$ is a neutral Tannakian category.
\end{thm}
\begin{proof}
 By Proposition \ref{first prop of differential modules} we see that
 $\Diff(A)$ is an abelian, monoidal and symmetric category with
 internal homomorphisms. By Proposition \ref{locally simple implies locally
 free} it is easy to see that $\Diff(A)$ is also rigid and that
 endomorphisms of the unit are either $0$ or isomorphisms, that
 is $\End_{\Diff(A)}(A)\subseteq A$ is a field. If $A$ is a
 differential $k$-algebra and $x$ a $k$-rational point, then we have that 
\[
k\subseteq \End_{\Diff(A)}(A)\subseteq k.
\]
Therefore $\Diff(A)$, with the fiber functor obtained via $x^*$, is a neutral Tannakian category.

\end{proof}

\subsection{Connections}

We now introduce a natural way of describing differential modules via connections.

\begin{defn}\label{def of connection}
Let $f\colon Y\arr S$ be a map of schemes and consider a surjective map of quasi-coherent sheaves $\Omega_{Y/S}\twoheadrightarrow\Omega$ such that the differential $\Omega_{Y/S}\to \Omega_{Y/S}^2$ induces $d^1\colon \Omega\to \Omega^2\coloneq\Omega\wedge\Omega$.
An $\Omega$-connection on an $\odi Y$-module $M$ is an $f^{-1}\sO_S$-linear map 
$$\nabla\colon \ \  M\arr M\otimes_{\sO_Y}\Omega$$
of sheaves satisfying the Leibniz rule, \emph{i.e.} $\nabla(am)=a\nabla(m)+m\otimes da$ for all sections $a$, $m$ on $\odi Y$, $M$ respectively over some open.

The connection $\nabla$ induces a map, 
\[
\nabla^1: M\otimes_{\sO_Y} \Omega\rightarrow M\otimes_{\sO_Y} \Omega^2
\]
defined by $\nabla^1(m\otimes
\omega)=\nabla(m)\wedge\omega+m\otimes\mathrm{d}\omega$ for all
sections $m$, $\omega$ on $M$, $\Omega$ respectively over some
open, where $\nabla(m)\wedge\omega$ is the image of
$\nabla(m)\otimes\omega$ under the canonical map
$$M\otimes_{\sO_Y}\Omega\otimes_{\sO_Y}\Omega\xrightarrow{\
\id\otimes \wedge\ }
M\otimes_{\sO_Y}\Omega^2.$$  The  map $\nabla^1$ is well-defined
thanks to  \cite[\href{https://stacks.math.columbia.edu/tag/07I0}{Tag
07I0}]{stacks-project}.

The connection $\nabla$ is called integrable if the 
composition $$M\xrightarrow{\nabla}M\otimes_{\sO_Y}\Omega\xrightarrow{\nabla^1}  M\otimes_{\sO_Y}\Omega^2$$ is zero. 

We denote the
category of integrable $\Omega$-connections in finitely presented $\mathcal{O}_Y$-modules by $\Conn(Y/S,\Omega)$.
\end{defn}

\begin{lem}
    Let $f\colon Y\arr S$ be a map of schemes and let $\sO_Y\xrightarrow{d}\Omega_{Y/S}^1\xrightarrow{d^1}\Omega_{Y/S}^2$ be the canonical differentials.
    Suppose that $\phi_1,\phi_2\in\textup{Der}(Y/S)$, and let $\varphi_1,\varphi_2$ be the corresponding  maps in
    $\Hom_{\sO_Y}(\Omega_{Y/S},\sO_Y)$. We denote $\varphi_1\wedge\varphi_2$ the map $\Omega_{Y/S}^2\arr \sO_Y$ sending $dx\wedge dy$ to $\phi_1(x)\phi_2(y)-\phi_2(x)\phi_1(y)$. Then $[\phi_1,\phi_2]$ corresponds to
    the homomorphism 
    $$\varphi_1\circ d\circ \varphi_2-\varphi_2\circ
    d\circ\varphi_1-(\varphi_1\wedge\varphi_2)\circ
    d^1\in\Hom_{\sO_Y}(\Omega_{Y/S},\sO_Y).$$
\end{lem}
\begin{proof}
If $\psi\colon \Omega_{Y/S}\arr \odi Y$ is the map in the statement, it clearly satisfies $\psi \circ d = [\phi_1,\phi_2]$. Thus one has to check that $\psi$ is $\odi Y$ linear. This is a direct computation which we omit.
\end{proof}

\begin{cor} \label{sub-Lie algebra}
 Let $f\colon Y\arr S$ be a map of schemes and $\Omega_{Y/S}\arr \Omega$ a
 quotient as in Definition \ref{def of connection}. Then the subsheaf 
 $\sHom_Y(\Omega,\odi Y)\subseteq \der(Y/S)$ 
 is a subsheaf of Lie algebras. 
\end{cor}

\begin{lem}\label{connection as stratification}
 Let $f\colon Y\coloneq\Spec A\arr S\coloneq\Spec R$ be a map of
 affine schemes and $\Omega_{Y/S}\to \Omega$ a quotient as in
 Definition \ref{def of connection}. Set
 $\Delta_A\coloneq\Hom_{\sO_Y}(\Omega,\odi Y)$. Assume moreover that $\Omega$ is
 locally free of finite type. Then $(A,\Delta_A)$
 is a differential ring over $R$. Moreover the functor
   \[
  \begin{tikzpicture}[xscale=4,yscale=-0.6]
      
    \node (A0_0) at (0, 0) {$F\colon\ \Conn(Y/S,\Omega)$};
    \node (A0_1) at (1, 0) {$\Diff(A)$};
    \node (A1_0) at (0, 1) {$(\tilde{M},\nabla_{\tilde{M}})$};
    \node (A1_1) at (1, 1) {$(M,\nabla_M)$};
    \path (A0_0) edge [->]node [auto] {$\scriptstyle{}$} (A0_1);
    \path (A1_0) edge [|->,gray]node [auto] {$\scriptstyle{}$} (A1_1);
  \end{tikzpicture}
  \]
which sends the $\mathcal{O}_Y$-module $\tilde{M}$ to the corresponding $A$-module $M$ and the $\Omega$-connection $\nabla_{\tilde{M}}$ to the map $\nabla_M$ defined on $\phi\in\Delta_A$ as
$\Hl^0(\tilde{M}\arrdi{\nabla_{\tilde{M}}}
\tilde{M}\otimes \Omega \arrdi{\id\otimes \phi} \tilde{M})$ is an equivalence of categories.
\end{lem}

\begin{proof} First we prove that the above functor
    $F\colon (\tilde{M},\nabla_{\tilde{M}})\mapsto (M,\nabla_M)$ induces an
    equivalence between the category of quasi-coherent $\Omega$-connections
    (not necessarily integrable) and the category of pairs
$(M,\nabla_M)$, where $M$ is an $A$-module and $\nabla_M$ is an
$A$-linear map $\Delta_A\to \End_{\Z}(M)$ satisfying the Leibniz
rule \eqref{eq:Leibniz rule} (not necessarily preserving the Lie
bracket). 

\textit{Full Faithfulness}. The faithfulness is clear. Now
suppose $\lambda\colon (M,\nabla_M)\to (N,\nabla_N)$ is a morphism in the target category. Then we get
directly a map $\lambda_\sO\colon \tilde{M}\to \tilde{N}$ between the corresponding $\mathcal{O}_Y$-modules, therefore we only have to
check that $\lambda_\sO$ is compatible with $\nabla_{\tilde{M}}$
and $\nabla_{\tilde{N}}$. We can check the compatibility Zariski
locally. We can localize both the $\Omega$-connections and the "not necessarily Lie-bracket
preserving differential modules" (Remark \ref{localize a differential
module}). The functor $F$ is compatible with the
localization, thus we are reduced to the case when $\Omega
=\sO_Y^{\oplus n}$ for some $n\in\N$. Then the map
$\nabla_{\tilde{M}}$ (resp. $\nabla_{\tilde{N}}$) becomes a map
of the form ${\tilde{M}}\to\prod_{i=1}^n{\tilde{M}}$ (resp.
${\tilde{N}}\to\prod_{i=1}^n{\tilde{N}}$). Let $p_i^M$ (resp.
$p_i^N$) be the $i$-th projection
$\prod_{i=1}^n{\tilde{M}}\to \tilde{M}$ (resp.
$\prod_{i=1}^n{\tilde{N}}\to \tilde{N}$). Since $\lambda$ is a map of differential modules, the map $\lambda_\sO$ is compatible with
$p_i^M\circ\nabla_{\tilde{M}}$ and
$p_i^M\circ\nabla_{\tilde{N}}$. Therefore, $\lambda_\sO$ is compatible with $\nabla_{\tilde{M}}$
and $\nabla_{\tilde{N}}$ by the universality of products of
modules.

\textit{Essential Surjectivity}. We cover $\Spec A$ by open
affines $\Spec A_{f_i}$. Suppose $\Omega$ is free over each
$\Spec A_{f_i}$, and suppose the claim holds when
$\Omega$ is free. Given $(M,\nabla_M)$ we get the localizations
$(M_i,\nabla_{M_i})$  on each $\Spec A_{f_i}$ and the corresponding quasi-coherent connections
$(\tilde{M}_i,\nabla_{\tilde{M}_i})$. Note that on
$U_{ij}\coloneq \Spec
A_{f_i}\bigcap \Spec A_{f_j}$ the sheaf $\Omega$ is also free,
and by the full faithfulness there is a unique isomorphism
$$(\tilde{M}_i,\nabla_{\tilde{M}_i})|_{U_{ij}}\xrightarrow{\
\ \ \ \simeq\ \ \ \
}(\tilde{M}_j,\nabla_{\tilde{M}_j})|_{U_{ij}}.$$
This allows to glue all $(\tilde{M}_i,\nabla_{\tilde{M}_i})$
together to get $(\tilde{M},\nabla_{\tilde{M}})$ which
corresponds to $(M,\nabla_M)$. We are therefore reduced to the case
    when $\Omega$ is free. In this case, we 
    can define $\nabla_{\tilde{M}}\colon \tilde{M}\arr
    \tilde{M}\otimes \Omega$ as follows. Choose a basis
    $s_1,\cdots ,s_n$ of $\Omega$ and let $f_1,\dots,f_n$ be
    its dual basis. We set
    $\nabla_{\tilde{M}}(m)=\sum_{i=1}^n\nabla_M(f_i)(m)\otimes s_i$ for all $m\in M$.

Now we come back to compare $\Conn(Y/S,\Omega)$ and $\Diff(A)$. To show
that the above equivalence induces the equivalence of these two
categories we just have to notice the formula in \cite[p.
179, last paragraph, 1.0.5]{Katz70} and the fact that
$\Delta_A\subseteq \Der(A/R)$ is a sub Lie algebra (Corollary \ref{sub-Lie
algebra}).  \end{proof}
\subsection{Crystalline site and crystals}
We recall here the general notion of small crystalline site and crystals on it. This was defined by Berthelot (\cite{Ber74}, \cite{BO78}). We use as our main reference for this theory  \cite[\href{https://stacks.math.columbia.edu/tag/09PD}{Tag 09PD}]{stacks-project} and \cite[\href{https://stacks.math.columbia.edu/tag/07GI}{Tag 07GI}]{stacks-project}.

\begin{defn}\label{dividedpowerscheme}\leavevmode
\begin{enumerate}
\item[-]{\cite[\href{https://stacks.math.columbia.edu/tag/07GU}{Tag 07GU}]{stacks-project}}
A divided power ring, or a PD-ring, is a triple $(A, I, \gamma)$ where $A$ is a ring, $I\subset A$ is an ideal, and  $\gamma=(\gamma_n)_{n\geq 1}$ is a divided power structure on $I$. A homomorphism of divided power rings $\varphi:(A,I,\gamma)\rightarrow (B,J,\delta)$ is a ring homomorphism $\varphi:A\rightarrow B$ such that $\varphi(I)\subset J$ and such that $\delta_n(\varphi(x))=\varphi(\gamma_n(x))$ for all $x\in I$ and $n\geq 1$. 
\item[-] {\cite[\href{https://stacks.math.columbia.edu/tag/07GI}{Tag 07GI}]{stacks-project}}.
A divided power scheme or a PD-scheme is the natural globalisation of a PD-ring. 
\item[-]
When we want to consider a homomorphism of PD-rings or PD-schemes, we will write it as a morphism of triples. On the other hand if $R$ is a ring an $R$-PD-ring is a PD-ring $(A, I, \gamma)$ where $A$ is an $R$-algebra (and the same for PD-schemes over $R$).
\end{enumerate}
\end{defn}

We fix a prime number $p$.
\begin{defn}\label{crysallinesiteoverS} \cite[\href{https://stacks.math.columbia.edu/tag/07IF}{Tag 07IF}]{stacks-project}
Let $\lS=(\lS, I, \gamma)$ be a PD-scheme such that $\lS$ is a $\Z_{(p)}$-scheme. Let $X$ be an $S=V(I)$-scheme, and we assume moreover that $p\in I$, 
\emph{i.e.} $S$ is killed by $p$. An object of the crystalline site $(X/\lS)_{\rm{crys}}$ is given by a triple $(U, \lT, \delta)$, 
where $U$ is a Zariski open of $X$, $\lT$ is an $\lS$-scheme,
$U\hookrightarrow \lT$ is a thickening of $\lS$-schemes defined
by a nil ideal $J$ and $(\lT,J,\delta)$ is a PD-scheme over
$(\lS, I, \gamma)$. We often denote $(U, \lT, \delta)$ simply by
$\lT$. Morphisms are defined in a natural way, and coverings are
defined using the Zariski topology on $\lT$.  We consider the
structure sheaf $\mathcal{O}_{X/\lS}$, defined by
$\mathcal{O}_{X/\lS}(\lT)\coloneq\Gamma(\lT, \mathcal{O}_\lT)$. 
\end{defn}

\begin{rmk}\label{directlimitofsites} Let  the notation be as in
    Definition \ref{crysallinesiteoverS} and set
    $\lS_n\coloneq \lS\times_{\Spec(\mathbb{Z})}
    \Spec(\mathbb{Z}/p^{n}\mathbb{Z})$. Then the crystalline site $(X/\lS)_{\rm{crys}}$ is the direct limit of the sites $(X/\lS_n)_{\rm{crys}}$.

\end{rmk}

\begin{rmk} We  use
    \cite[\href{https://stacks.math.columbia.edu/tag/09PD}{Tag
    09PD}]{stacks-project} as the main reference. Here we want to stress the compatibility of Definition \ref{crysallinesiteoverS} with more classical references. 
\begin{enumerate}
\item If $\lS$ is killed by a power of $p$, then the site
    defined in Definition \ref{crysallinesiteoverS} is the same
    as the crystalline site defined in \cite[p. 5.1]{BO78}, with the hypothesis that $p\in I$.

\item When $\lS=\Spec R$ is the spectrum of a Noetherian ring
    $R$ which is complete for the $I$-adic topology, and if
    $p\in I$, then the crystalline site $(X/\lS)_{\rm{crys}}$ of
    Definition \ref{crysallinesiteoverS} is equivalent to the
    site $\mathrm{Cris}(X/\hat{S})$ defined in \cite[7.17]{BO78}
(with $P=I$), where $\hat{S}\coloneq\textup{Spf}\,  R$ for the $I$-adic
topology. 
\item Shiho, in \cite{Shi07I}, developed a theory of relative crystalline cohomology for log schemes. He supposes that $I=p$ and (here we are in the simplified case where all the log structures are trivial) he generalised the situation (2) to the case where $\lS$ is a $p$-adic formal scheme separated and topologically of finite type over  $W$.

\end{enumerate}
\end{rmk}
\begin{defn}\label{crystalsoverS}
An $\mathcal{O}_{X/\lS}$-module $E$ on the site $(X/\lS)_{\mathrm{crys}}$ is called a crystal if every morphism $\varphi:\lT\rightarrow \lT'$ in $(X/\lS)_{\mathrm{crys}}$ induces an isomorphism $\varphi^*(E_{\lT'})\rightarrow E_\lT$, where we denote with $E_{\lT'}$ (resp. $E_{\lT})$ the Zariski sheaf on $\lT'$ (resp. on $\lT$) induced by $E$. A crystal is said to be of finite presentation if for every $\lT\in (X/\lS)_{\mathrm{crys}}$ the $\mathcal{O}_\lT$-module $E_{\lT}$ is of finite presentation. The category of crystals of finite presentation on $(X/\lS)_{\mathrm{crys}}$ is denoted by $\mathrm{Crys}(X/\lS)$.
\end{defn}

For any commutative diagram 
\[
\xymatrix{
X'\ar[r]^h\ar[d]^{g'}&X\ar[d]^g\\
\lS'\ar[r]^u&\lS
}
\]
where $u$ is a PD-morphism, we obtain a morphism of ringed topoi $h_{\rm{crys}}=(h_{\rm{crys}}^*, h_{\rm{crys}*})$ (\cite[\href{https://stacks.math.columbia.edu/tag/07KL}{Tag 07KL}]{stacks-project}). It is known that if $E$ is a crystal in $\mathrm{Crys}(X/\lS)$, then $h_{\rm{crys}}^*(E)$ is a crystal in $\mathrm{Crys}(X'/\lS')$ (\cite[Corollaire 1.2.4]{Ber74} and Remark \ref{directlimitofsites}). 


\begin{conv}\label{situation0}
      Let the hypothesis and notation be as in Definition \ref{crysallinesiteoverS}. Suppose moreover that we have a commutative diagram 
      \begin{equation}\label{situation}
      \xymatrix{X\ar@{^{(}->}[rr]^-i\ar[d]&&\lX\ar[d]^-f\\ S\ar[rr]&& \lS}
      \end{equation}
       in which $f$ is smooth and every scheme is affine: $X=\Spec C,\lX=\Spec P, S=
       \Spec A/I, \lS= \Spec A$. The map $i$ in the above diagram is a
       closed immersion defined by an ideal $J\subseteq P$ (in
       particular $IP\subset J$). 
       Let $\lD_{P,\gamma}\coloneq\Spec D_{P,\gamma}$ be the PD-envelope of $i:X\hookrightarrow \lX$ with
    respect to $(\lS, I,\gamma)$ and let $D$ be the $p$-adic
    completion of $D_{P,\gamma}$. Set $\lD\coloneq\Spec D$, $A_n\coloneq A/p^n$, 
    $\lS_n\coloneq \Spec A_n$, $P_n\coloneq
    P\otimes_AA_n$ and $\lX_n\coloneq \Spec P_n$. Let $\lD_{P_n,\gamma}\coloneq\Spec
    D_{P_n,\gamma}$ be the PD-envelope of $X\hookrightarrow
    \lX_n$ with
    respect to $(\lS, I,\gamma)$. Thanks to \cite[\href{https://stacks.math.columbia.edu/tag/07KG}{Tag
    07KG}]{stacks-project} we have $D=\varprojlim_{n\in \N}D_{P_n,\gamma}$ as
    PD-rings.
\end{conv}

\subsubsection{Crystals and connections over complete PD-envelopes}

    We denote by $\Omega_{D}$ the $p$-adic completion of the
    module of PD-differentials $\Omega_{D_{P,\gamma}/A,\bar{\gamma}}$
    (see
    \cite[\href{https://stacks.math.columbia.edu/tag/07HQ}{Tag
    07HQ}]{stacks-project}). Notice that $\Omega_D$ is a finite projective
    ${D}$-module: indeed
    \[\Omega_{D_{P,\gamma}/A,\bar{\gamma}}\simeq
    \Omega_{P/A}\otimes_P D_{P,\gamma}\](see
    \cite[\href{https://stacks.math.columbia.edu/tag/07HW}{Tag
    07HW}]{stacks-project}) and, when we take the $p$-adic
    completion, the left hand side, by definition, becomes
    $\Omega_{D}$, while the right hand side is isomorphic to
    $\Omega_{P/A}\otimes_P D$ because $\Omega_{P/A}$ is a finite projective $P$-module. Therefore 
     \[
        \Omega_D \simeq
        \Omega_{P/A}\otimes_P D
    \]
    which is a finite projective $D$-module.
 We denote by $\Omega_\lD$ the sheaf on $\lD$ associated to $\Omega_D$.

\begin{rmk}\label{omegasurjective}
Thanks to \cite[\href{https://stacks.math.columbia.edu/tag/07KG}{Tag
    07KG}]{stacks-project} we have 
\[
        \Omega_D\otimes_AA_n \simeq
        \Omega_{P_n/A}\otimes_{P_n} D_{P_n,\gamma}\simeq
        \Omega_{D_{P_n,\gamma}/A,\bar{\gamma}}
    \] for $n$ large. This allows us to
 construct a map 
 \[
    \Omega_{D/A}\ \arr \Omega_{D}
    \]
which is split surjective. Indeed, the section 
\[
 \Omega_D\simeq\Omega_{P/A} \otimes_P D \rightarrow \Omega_{D/A}
\]
is given by the
extension of scalars of the natural map $\Omega_{P/A}\arr
\Omega_{D/A}$ along the  map $P\to D$. 
\end{rmk}

%

\begin{defn}

    In the situation of \textbf{Setting} \ref{situation0},
    we denote by $\Conn(X/\lS, i,f)$ the full subcategory of
the category    $\Conn(\lD/\lS,\Omega_\lD)$ consisting of integrable $\Omega_{\lD}$-connections 
($\tilde{M},\nabla_{\tilde{M}}$), where $M$
    is a finitely
presented  $p$-adically complete $D$-module.
    
    \end{defn}

\begin{rmk}\label{When the diagram is catesian}\leavevmode
    \begin{enumerate}
        \item If $\lD$ is Noetherian, then
            $\Conn(X/\lS,i,f)=\Conn(\lD/\lS,\Omega_\lD)$ because
            in this case any finitely presented $D$-module is
            $p$-adically complete. 

    \item If $M$ is a $p$-adically complete $D$-module, the module $
        M\otimes_{D}\Omega_{D}$ is $p$-adically complete because $\Omega_D$ is
        a finite projective $D$-module. 
    In particular the connections defined above agree with the pairs considered in \cite[\href{https://stacks.math.columbia.edu/tag/07J7}{Tag 07J7}]{stacks-project}.
    
    \item If the diagram in \eqref{situation} is Cartesian, then
        the PD-structure $\gamma$ extends to $X\hookrightarrow
        \lX$
        (\cite[\href{https://stacks.math.columbia.edu/tag/07H1}{07H1}]{stacks-project}),
        and $\lD_{P,\gamma}=\lX$. Indeed, since the diagram is cartesian,
        $IP=J$ and $(P,
        IP)$ verifies the universal property of
        the PD-envelope. With these hypothesis we get that
        $\Omega_{D_{P,\gamma}/A,\bar{\gamma}}=\Omega_{P/A}$ (see
        \cite[\href{https://stacks.math.columbia.edu/tag/07HW}{Tag
        07HW}]{stacks-project}). Therefore the $p$-adic completions are isomorphic 
       
      \[
\widehat{\Omega_{D/A}} \simeq \widehat{\Omega_{P/A}}\simeq \Omega_D.
    \]
Moreover
\[
\Hom(\Omega_D,D)=\Der(D/A);
\] indeed a map
from $\Omega_{D/A}$ to a $p$-adically complete module factors
through $\Omega_D$. We remark that any derivation in $\Der(D/A)$ is
$\Z$-linear, hence it is automatically $p$-adically continuous.
    \item If we have another commutative diagram
\begin{equation*}
    \xymatrix{X'\ar@{^{(}->}[rr]^-{i'}\ar[d]&&\lX'\ar[d]^-{f'}\\
    S'\ar[rr]&& \lS'}
      \end{equation*}
 mapping to the original one, there is
an induced map $D'\arr D$ which yields a map
$\Omega_{D}\otimes_D D'\arr \Omega_{D'}$. Via this map we
obtain a functor $\Conn(X/\lS, i,f)\arr
\Conn(X'/\lS',i',f')$.     \end{enumerate}
\end{rmk}

\subsubsection{Topologically quasi-nilpotent connections}
\begin{defn}\label{quasi-nilpotent-affine-space0}
    In the situation of \textbf{Setting} \ref{situation0} where
    \eqref{situation} is cartesian, a connection 
$(\tilde{M},\nabla_{\tilde{M}})\in\Conn(X/\lS,i,f)$ is called
\textit{topologically quasi-nilpotent} if for all $n\geq 1$ its reduction
modulo $p^n$ is quasi-nilpotent in the sense of \cite[Definition 4.10, Remark
4.11]{BO78}.

We denote by    $\QNCf(X/\lS,i,f)$ the full subcategory of
$\Conn(X/\lS,i,f)$ consisting of
topologically quasi-nilpotent connections.
  \end{defn}



\begin{thm}\label{cry equivalent to quasi-nil}
    Let $X,\lX,\lS$ be as in Definition \ref{quasi-nilpotent-affine-space0}.
    Then there is a fully faithful additive tensor  functor
    \[
    \Crys(X/\lS)\arr \Conn(X/\lS,i,f)
\] whose essential image is $\QNCf(X/\lS,i,f)$. Moreover, the above
functor is functorial with respect to the diagram
\eqref{situation}.
\end{thm}
 
\begin{proof}     
    Given $E\in \Crys(X/\lS)$,  we take its restriction
    $E_n \in
    \Crys(X/\lS_n)$, obtaining $(\tilde{M}_n,\nabla_{\tilde{M}_n}) \in \Conn(X/\lS_n,i_n,f_n)$
    by \cite[Theorem 6.6]{BO78}. Here the $P_n$-module $M_n$ is
    $\Hl^0(E_{\lX_n})$. There are transition maps $\phi_n\colon
    M_{n+1}\to M_n$ which are horizontal, that is they preserve
    the connections. Since $E$ is a crystal, we have $M_{n+1}/p^nM_{n+1} \simeq M_n$.
    
The limit $M\coloneq \varprojlim_{n\in\N^+} M_n$ is a $D$-module since $D/p^nD=P_n$ and  $M/p^nM=M_n$
by
\cite[\href{https://stacks.math.columbia.edu/tag/09B8}{09B8}]{stacks-project}.
Moreover, $M$ also comes with a connection.
This association defines the functor $\Crys(X/\lS) \to \Conn(X/\lS,i,f)$, which is easily seen to be linear and to preserve the tensor product.

The full faithfulness and the claim about the essential image
follow from the corresponding statements in the 
$p^n$-torsion case (see e.g. \cite[Corollary 6.8]{BO78} or \cite[Théorème 1.6.5, p.
247]{Ber74}).
 
\end{proof}

\begin{rmk}\label{being qncf over lS = being qncf
    over lSn}\leavevmode
\begin{enumerate}

    \item The naturality of the functor in Theorem \ref{cry equivalent
        to quasi-nil} indicates that the pullback of a
        topologically
        quasi-nilpotent connection is topologically
        quasi-nilpotent.
    \item Directly from the definition one sees that
        $(\tilde{M},\nabla_{\tilde{M}})\in \Conn(X/\lS,i,f)$ belongs to
        $\QNCf(X/\lS,i,f)$ if and only if its pullback to
        $X/\lS_n$ belongs to $\QNCf(X/\lS_n, i_n, f_n)$ (see
        Remark \ref{directlimitofsites} for the notation) for some
        $n\in\N$.
\end{enumerate}  
\end{rmk}

\begin{lem}\label{subcategory of quasi-nilpotent}
  Suppose that we are in the situation of Definition
  \ref{quasi-nilpotent-affine-space0}. Then
  $$\QNCf(X/\lS,i,f)\subseteq \Conn(X/\lS,i,f)$$ is a full
  subcategory closed under taking subobjects, quotients,  tensor products and internal homomorphisms.
\end{lem}
\begin{proof}
    Directly from  Definition \ref{quasi-nilpotent-affine-space0} it is clear that subojects and quotient
    objects of topologically quasi-nilpotent connections are
    topologically quasi-nilpotent. We still have to show that if
    $(\tilde{E},\nabla_{\tilde{E}})$ and
    $(\tilde{F},\nabla_{\tilde{F}})$ are  topologically
    quasi-nilpotent connections, then their tensor product and
    their Hom are topologically quasi-nilpotent. This follows by
    checking the following relations for all $D\in \Der(\lD/\lS)$: 
 $[\nabla_{E\otimes F}(D)]^n(e\otimes f)$ is
    $$
    D^n(e\otimes f) = D^{n}(e)\otimes f+\cdots+\binom{n}{r}D^{n-r+1}(e)\otimes D^r(f)+\cdots+e\otimes D^n(f)
    $$
    and $[\nabla_{\Hom(E,F)}(D)]^n(\phi)$ is 
    \[
        \nabla_F(D)^n \circ \phi+\cdots+(-1)^r\binom{n}{r}\nabla_F(D)^{n-r}\circ
    \phi\circ \nabla_E(D)^r+\cdots+(-1)^n\phi\circ
    \nabla_E(D)^n.\qedhere
\]
 \end{proof}

\subsubsection{The situation when \eqref{situation} is cartesian}
\label{1.3cartesian} 


\begin{lem}\label{I-adic=p-adic} 

    Let $(\lS=\Spec A,I,\gamma)$ be an affine PD-scheme over $\mathbb{Z}_{(p)}$ such that $p\in I$. As above set  $S = \Spec A/I$,  $A_n=A/p^n$ and  $\lS_n=\Spec A_n$ for all $n\in \mathbb{N}$.

The closed embedding $S\hookrightarrow \lS_1$ is a locally nilpotent thickening, that is $I/pA$ is a nil ideal in $A/p$. In particular, if the ideal $I$ is finitely generated, then the closed embedding $S\hookrightarrow \lS_1$ is a nilpotent thickening. 
\end{lem}
\begin{proof}  The ideal $I$ has a PD-structure and therefore $p!\gamma_p(x)=x^p$ for all $x\in I$, so that $x^p\in pA$ as required.
\end{proof}

\begin{rmk}\label{1.3cartesian} Let $\lS=(\lS, I, \gamma)$ be a PD-scheme as in Lemma \ref{I-adic=p-adic}. Suppose moreover that $I$ is finitely generated. Let $g\colon X=\Spec C \to S=V(I)$ be a smooth map. 
    Under the assumptions of Lemma \ref{I-adic=p-adic}, we can build up a
diagram
\eqref{situation} out of the given map  $g\colon X\to S$ and the closed immersion $S\hookrightarrow\lS$ such that it is a cartesian diagram. Indeed, 
by
\cite[\href{https://stacks.math.columbia.edu/tag/07M8}{07M8}]{stacks-project} we can lift $g$ to a
smooth affine
map $f\colon \lX=\Spec P \to \lS=\Spec A$
not necessarily uniquely along $S\hookrightarrow \lS$.
Note that by \cite[Theorem 8.5.9]{ill05} the lifts of $g$ along
$S\hookrightarrow \lS_1$ and
$\lS_n\hookrightarrow\lS_{n+1}$ are unique. Thanks to Remark \ref{When the diagram is catesian} (3) and
the uniqueness of the lift to $\lS_n$ for all $n$, 
the spectrum $\lD$ of the $p$-adically completed PD-envelope $D$, which is the
$p$-adic completion of $P$, does not depend on the lift 
$f\colon \lX\to\lS$ we chose for $g$. 
\end{rmk}

\begin{defn}
Let $\lS=(\lS, I, \gamma)$ be a PD-scheme as in Lemma \ref{I-adic=p-adic} and we assume that $I$ is finitely generated.  Let $g\colon X\to S=V(I)$ be a smooth map. We construct a cartesian diagram as in Remark \ref{1.3cartesian}. As observed in Remark \ref{1.3cartesian}, 
the category $\Conn(X/\lS,i,f)$ does not depend on the choice of $f$ and
$i$ such that (\ref{situation}) is cartesian, so, in this case, we will just
write $\Conn(X/\lS)$ instead of $\Conn(X/\lS,i,f)$. Thanks to 
Theorem \ref{cry equivalent to quasi-nil} the full
subcategory $\QNCf(X/\lS,i,f)$ does not depend on the choice of such $f$ and $i$ either, thus we will write $\QNCf(X/\lS)$ instead
 of $\QNCf(X/\lS,i,f)$ when the conditions of
 Lemma \ref{I-adic=p-adic} are met.
\end{defn}

\begin{lem} Let $(\lS=\Spec A,I,\gamma)$ be as in in Lemma \ref{I-adic=p-adic}, and let $g\colon X\to S=V(I)$ be a smooth map.
    If $\lS$ is Noetherian, then we have
    $$\Conn(X/\lS)=\Conn(\lD/\lS,\Omega_\lD).$$  Therefore, the
    category
    $\Conn(X/\lS)$ is an abelian, symmetric monoidal category with internal homomorphisms.
\end{lem}
\begin{proof}
If $\lS$ is Noetherian, then $D$ is Noetherian and
$p$-adically complete, so every finitely presented $D$-module is $p$-adically
complete. The last claim follows from Lemma \ref{connection as stratification}
and Proposition \ref{first prop of differential modules}.
\end{proof}

\begin{lem}\label{connections in char zero are locally simple}
Let $(\lS=\Spec A,I,\gamma)$ be as in in Lemma \ref{I-adic=p-adic}, and let $g\colon X=\Spec C \to S=V(I)$ be a smooth map. Suppose moreover that $\lS=\Spec
A$, where $A$ is a complete DVR  of mixed characteristic $(0,p)$
with perfect residue field $k$ and fraction
field $K$. 

If $X$ is connected, then the rings $D$ (see \textbf{Setting} \ref{situation0}) and $D\otimes_A  K$ are regular domains and
$(D\otimes_A  K,\Der(D/A)\otimes_A K)$ is a locally simple differential ring.
\end{lem}
\begin{proof} We lift, as in Remark \ref{1.3cartesian}, the smooth map $g\colon X=\Spec C \to S=V(I)$ to a smooth map $f\colon \lX=\Spec P \to \lS=\Spec A$.

We first show that $D$ is a regular domain. Thanks to
    \cite[\href{https://stacks.math.columbia.edu/tag/07qw}{07QW}]{stacks-project}
    the ring $P$ is excellent, so it is a G-ring. According to
    \cite[\href{https://stacks.math.columbia.edu/tag/0AH2}{0AH2}]{stacks-project}
     the completion $P\to D$ is a regular map (\emph{i.e.} has geometrically regular fibers). Taking into account \cite[\href{https://stacks.math.columbia.edu/tag/031E}{031E}]{stacks-project} and the fact that $A\to P$ is regular by construction, we can conclude that $D$ is a regular ring.
     
     In order to conclude that $D$ is also a domain, it is enough to show that $\lD=\Spec D$ is connected.
Since the ideal $I$ is finitely generated,  by Lemma \ref{I-adic=p-adic} the maps $S\hookrightarrow \lS_n$ are  nilpotent thickenings as well as the maps $X\hookrightarrow \lX_n=\Spec D/p^n$ because the diagram 
    \begin{equation*}
      \xymatrix{X\ar@{^{(}->}[rr]\ar[d]&&\lX\ar[d]\\ S\ar[rr]&& \lS}
      \end{equation*}
    is cartesian. Therefore $\lX_n=\Spec D/p^n $  is connected for all
    $n\in\N$, because $X$ is connected by hypothesis. 
    
    In particular if $a \in D$ is an idempotent element, then
$$a_n\coloneq a\mod p^{n}D$$ is either $0$ or $1$ in
    $D/p^n$. As 
    $X$ is non-empty, none
    of those $(D/p^n)$'s is a zero ring, so $a_n\in D/p^n$ has to
    be $0$ for all $n$ or $1$ for all $n$. Thus
    $a=0$ or $1$ in $D$, which implies that $\lD$ is connected.      

From the fact that $D$ is a regular domain we deduce that its localization $D\otimes_A  K$ is a regular domain as well.

Thus it remains to prove that
$(D\otimes_A  K,\Der(D/A)\otimes_A K)$ is a locally simple differential ring. By \cite[Lemma 1.19]{Ogus84} and its proof we see that for any closed point $x\in \lD\times_\lS \Spec K$ the map 
 \[
 m_x/m_x^2\arr \Omega_D\otimes k(x)
 \]
 is injective, where $m_x$ and $k(x)$ are the maximal ideal and the residue field of $x$ respectively. Applying $\Hom_{k(x)}(-,k(x))$ and recalling that $\Omega_D$ is locally free  we obtain a surjective map
    $$\Hom_{D}(\Omega_D,D)\otimes_{D}k(x)\arr
    \Hom_{k(x)}(m_x/m_x^2,k(x)).$$
Since $\Hom(\Omega_D,D)=\Der(D/A)$ the result follows from \cite[Proposition 1.2.3]{Ked12}.\end{proof}

\begin{thm}\label{diff modules and connections are tannakian}
Let $(\lS=\Spec A,I,\gamma)$ be an affine PD-scheme over $\mathbb{Z}_{(p)}$ such that $p\in I$ and let $g\colon X=\Spec C \to S=\Spec(A/I)$ be a smooth map. Suppose moreover that $\lS=\Spec
A$, where $A$ is a complete DVR  of mixed characteristic $(0,p)$
with perfect residue field $k$ and fraction
field $K$. If $X$ is connected, then we have a diagram of Tannakian categories
    \[
  \begin{tikzpicture}[xscale=5.8,yscale=-1.2]
    \node (A0_0) at (0.1, 0) {$\QNCf(X/\lS)\otimes_A K$};
    \node (A0_1) at (1, 0) {$\Conn(\lD/\lS,\Omega_\lD)\otimes_A K$};
    \node (A0_2) at (2, 0) {$\Conn(\lD\otimes_A
        K/K,\Omega_\lD\otimes_A K)$};
    \node (A1_1) at (1, 1) {$\Diff(D,\Der(D/A))\otimes_A K$};
    \node (A1_2) at (2, 1) {$\Diff(D\otimes_A
        K,\Der(D/A)\otimes_A K)$};
    \path (A0_0) edge [->]node [auto] {$\scriptstyle{}$} (A0_1);
    \path (A0_1) edge [->]node [above,rotate=-90] {$\scriptstyle{\simeq}$} (A1_1);
    \path (A1_1) edge [->]node [auto] {$\scriptstyle{}$} (A1_2);
    \path (A0_2) edge [->]node [above,rotate=-90] {$\scriptstyle{\simeq}$} (A1_2);
    \path (A0_1) edge [->]node [auto] {$\scriptstyle{}$} (A0_2);
  \end{tikzpicture}
  \]
where all the functors are fully faithful tensor exact functors.
\end{thm}
\begin{proof}
 The two vertical equivalences come from Lemma \ref{connection as
 stratification} since $\Der(D/A)=\Hom_D(\Omega_D,D)$ and that
 $\Omega_D$ is locally free. Notice that $\lD=\Spec D$ is
 Noetherian because $D$ is a completion of an affine smooth
 $A$-algebra. In particular the horizontal arrows on the right
 are fully faithful, exact, tensorial and preserve internal homomorphisms
 thanks to Lemma \ref{localization}.
 The left horizontal arrow is fully faithful, exact, tensorial and preserves
 internal homomorphisms by Lemma \ref{subcategory of quasi-nilpotent}.
 
 By Theorem \ref{locally simple implies Tannakian} and Lemma \ref{connections
 in char zero are locally simple} we can conclude that
 $\Diff(D\otimes_A K,\Der(D/A)\otimes K)$ is a Tannakian category. From this it easily follows that for all other categories there exists a fiber functor and the endomorphisms of the trivial object form a field. The rigidity of those categories also follows. Indeed we must check that for all objects $M,N$ in one of those categories the natural arrow
 \[
 M^\vee \otimes N \arr \textup{\underline{Hom}}(M,N)
 \]
 where $\textup{\underline{Hom}}(M,N)$ denotes the internal Hom, is an isomorphism. Because all functors preserves internal
 homomorphisms and tensor product, this morphism become an
 isomorphism in $\Diff(D\otimes_A K,\Der(D/A)\otimes K)$ and, because all  functors are fully faithful, this morphism has to be an isomorphism.
\end{proof}

\subsection{Crystalline fundamental group}

In this section we consider the following situation.  Let $k$ be
a perfect field of characteristic $p>0$, and let $W$ be the ring of
Witt vectors of $k$. Set $\lW\coloneq \Spec W$. We denote by $\gamma$ the canonical
$\mathrm{PD}$-structure on $pW$, $K$ the fraction field of
$W$. Set $W_n\coloneq W/p^{n}W$ and $\lW_n\coloneq\Spec W_n$. We denote by $\gamma_n$ the
induced $\mathrm{PD}$-structure on $pW_n$. The base PD-scheme
$(\lS,I,\gamma)$ is
$(\lW, pW, \gamma)$, and $S=\Spec k$.

\begin{defn}\label{isocrystals}
Let $X$ be a scheme over $k$. We denote by $I_{\rm{crys}}(X/\lW) $ the category of finitely presented isocrystals. This is the category  $\Crys(X/\lW)$ up to
isogeny, \emph{i.e.} the category whose objects are exactly those in
$\Crys(X/\lW)$ and whose morphisms are obtained inverting the multiplication by $p$. Thus we have a natural functor
$$\Crys(X/\lW)\arr I_{\crys}(X/\lW)$$
which is the identity on objects. To distinguish objects
in $\Crys(X/\lW)$ from those in $I_{\crys}(X/\lW)$ we denote by
$K\otimes E$ the image of $E\in\Crys(X/\lW)$ under the above
functor, and we say that $E$ is a lattice for the isocrystal $\mathcal{E}$ if $K\otimes E\cong \mathcal{E}$.
\end{defn}

The main result of the section is the following
\begin{thm} \label{isocrystannaka}
If $X$ is a smooth, quasi-compact and connected $k$-scheme, then the category $I_{\rm{crys}}(X/\lW)$ is a Tannakian category over a field
    extending $K$. 
    
    If $Y$ is another smooth, quasi-compact and connected
    $k$-scheme with a map $Y\arr X$, then the pullback
    $I_{\rm{crys}}(X/\lW)\arr I_{\rm{crys}}(Y/\lW)$ is an exact tensor functor.
    Moreover $I_{\rm{crys}}(\Spec k/\lW) = \Vect(K)$ and, if
    $x\colon \Spec k\arr X$ is a rational point, then $I_{\rm{crys}}(X/\lW)$ is a neutral $K$-Tannakian category via $x^*\colon I_{\rm{crys}}(X/\lW)\arr I_{\rm{crys}}(\Spec k/\lW)=\Vect(K)$.
\end{thm}

\begin{defn}\label{crystalline fundamental group}
Let $X$ be a smooth, quasi-compact and connected $k$-scheme with a rational
point $x\in X(k)$. We define $\pi_1^{\rm{crys}}(X/\lW, x)$ as the Tannaka dual
of the neutral Tannakian category $I_{\rm{crys}}(X/\lW)$ endowed with the fiber
functor $x^*$ (see Theorem \ref{isocrystannaka}).
\end{defn}

\begin{rmk}
The prounipotent completion of the group scheme defined in Definition \ref{crystalline fundamental group} has been defined and studied by Shiho in \cite{Shi00} and \cite{Shi02} (in the more general situation of log schemes).
\end{rmk}

\begin{lem}\label{affine lift general}
 Let $R$ be a complete Noetherian ring with respect to an ideal
$I\subseteq R$, and set $Z\coloneq\Spec R/I$, $\lZ\coloneq\Spec
R$. Consider also a smooth affine map  $V \arr Z$. We denote by $(-)_n$ the base change to $\lZ_n=\Spec(R/I^n)$. Then:
\begin{enumerate}
\item There exists a smooth affine map $\widetilde V=\Spec \widetilde D\to \lZ$ lifting $V\to Z$.
 \item There exists an affine and flat map $V_\lZ=\Spec D \to
     \lZ$ lifting $V\to Z$ such that $D$ is an $I$-adically
     complete ring. We can choose as $D$ the $I$-adic completion
     of an $R$-algebra $\widetilde D$ as in $(1)$.
     Moreover, $V_\lZ$ is a Noetherian scheme and all $(V_\lZ)_n \to \lZ_n$ are smooth.
 \item If $V_\lZ\to \lZ$ and $V'_\lZ\to \lZ$ are two lifts as in $(2)$ then there exists a (not necessarily unique) $\lZ$-isomorphism $V_\lZ\to V'_\lZ$ lifting $\id_V\colon V\to V$.
\end{enumerate}
\end{lem}
\begin{proof}
 $(1)$ This is 
\cite[\href{https://stacks.math.columbia.edu/tag/07M8}{Tag
07M8}]{stacks-project}.

$(2)$ Let $D$ be the $I$-adic
completion of $\widetilde D$ and set
$V_\lZ \coloneq \Spec D$.
 By
\cite[\href{https://stacks.math.columbia.edu/tag/05GH}{Tag
05GH}]{stacks-project} and
\cite[\href{https://stacks.math.columbia.edu/tag/0912}{Tag
0912}]{stacks-project} the ring $D$ is $ID$-adically complete,
Noetherian, $D/ID = \widetilde D/I\widetilde D$ and
$\widetilde D\to D$ is flat, so that $V_\lZ \to \lZ$ is
flat as well.

$(3)$ It is enough to find a system of compatible $\lZ_n$-maps $\phi_n\colon (V_\lZ)_n \to (V'_\lZ)_n$ with $\phi_0=\id_V$ (and thus automatically isomorphisms). Consider the diagram
  \[
  \begin{tikzpicture}[xscale=2.5,yscale=-1.2]
    \node (A0_0) at (0, 0) {$(V_\lZ)_n$};
    \node (A0_1) at (1, 0) {$U$};
    \node (A0_2) at (2, 0) {$(V_\lZ)_{n+1}$};
    \node (A1_0) at (0, 1) {$(V'_\lZ)_n$};
    \node (A1_1) at (1, 1) {$(V'_\lZ)_{n+1}$};
    \node (A2_0) at (0, 2) {$\lZ_n$};
    \node (A2_1) at (1, 2) {$\lZ_{n+1}$};
    \path (A0_2) edge [->,dashed]node [auto] {$\scriptstyle{\phi_{n+1}}$} (A1_1);
    \path (A0_0) edge [->]node [auto] {$\scriptstyle{}$} (A0_1);
    \path (A0_1) edge [->]node [auto] {$\scriptstyle{\alpha}$} (A1_1);
    \path (A1_0) edge [->]node [auto] {$\scriptstyle{}$} (A1_1);
    \path (A0_2) edge [->,dashed]node [auto] {$\scriptstyle{\beta}$} (A0_1);
    \path (A0_2) edge [->,bend right=30]node [auto] {$\scriptstyle{}$} (A2_1);
    \path (A1_0) edge [->]node [auto] {$\scriptstyle{}$} (A2_0);
    \path (A0_0) edge [->,bend right=40]node [auto] {$\scriptstyle{}$} (A0_2);
    \path (A1_1) edge [->]node [auto] {$\scriptstyle{}$} (A2_1);
    \path (A0_0) edge [->]node [auto] {$\scriptstyle{\phi_n}$} (A1_0);
    \path (A2_0) edge [->]node [auto] {$\scriptstyle{}$} (A2_1);
  \end{tikzpicture}
  \]
where $\alpha\colon U\to (V'_\lZ)_{n+1}$ is any flat lift of $\phi_n$, which exists by $(2)$ because $\phi_n$ is an isomorphism and thus it is smooth. Since $\lZ_n$ is affine, by \cite[Theorem
8.5.9, pp. 213-214]{ill05} we can find the dashed
$\lZ_{n+1}$-isomorphism $\beta\colon (V'_\lZ)_{n+1} \to U$
making the above diagram commutative. The choice
$\phi_{n+1}=\alpha\circ\beta$ yields the desired lifting of $\phi_n$.
\end{proof}

\begin{lem} \label{affine lift}
    Let  $X$ be a smooth affine 
    scheme over $k$. Then:
    \begin{enumerate}
    \item There exists a smooth affine map $\widetilde X = \Spec \widetilde B \to \lW$ lifting $X\to \Spec k$.
     \item There exists a flat and affine $\lW$-scheme $
         X_\lW=\Spec B \to \lW$ lifting $X\to \Spec k$ and such
         that $B$ is $p$-adically complete. We can choose as $B$
         the $p$-adic completion of a $W$-algebra $\widetilde B$
         as in $(1)$. Moreover, $X_\lW$ is a Noetherian scheme and all maps $(X_\lW)_n=\Spec B/p^{n}B\to \lW_n$ are smooth.
     
    \item If $f\colon Y\arr X$ is a smooth affine map over $k$
        and $X_\lW, Y_\lW \to \lW$ are the complete lifts of $X,Y$ as in $(2)$ respectively then there exists
        a flat map $f_\lW\colon Y_\lW \arr X_\lW$ lifting
        $f\colon Y\arr X$. Moreover,  all $(f_\lW)_n\colon (Y_\lW)_n \to (X_\lW)_n$ are smooth.
        
    \item If $X_\lW\to \lW$ and $X'_\lW\to \lW$ are two lifts as in $(2)$ then there exists a $\lW$-isomorphism $X_\lW \to X'_\lW$ lifting $\id_X\colon X\to X$.
    
    \item If $f_\lW,f'_\lW\colon Y_\lW\to X_\lW$ are two lifts as in $(3)$ then there exists an automorphism $\sigma$ of $Y_\lW$ fitting in the diagram
      \[
  \begin{tikzpicture}[xscale=1.5,yscale=-1.2]
    \node (A0_0) at (0, 0) {$Y$};
    \node (A0_1) at (1, 0) {$Y$};
    \node (A0_2) at (2, 0) {$X$};
    \node (A0_3) at (3, 0) {$\Spec k$};
    \node (A1_0) at (0, 1) {$Y_\lW$};
    \node (A1_1) at (1, 1) {$Y_\lW$};
    \node (A1_2) at (2, 1) {$X_\lW$};
    \node (A1_3) at (3, 1) {$\lW$};
    \path (A0_1) edge [->]node [auto] {$\scriptstyle{}$} (A1_1);
    \path (A0_0) edge [->]node [auto] {$\scriptstyle{\id_Y}$} (A0_1);
    \path (A0_1) edge [->]node [auto] {$\scriptstyle{}$} (A0_2);
    \path (A1_0) edge [->,dashed]node [auto] {$\scriptstyle{\sigma}$} (A1_1);
    \path (A0_3) edge [->]node [auto] {$\scriptstyle{}$} (A1_3);
    \path (A1_1) edge [->]node [auto] {$\scriptstyle{f_\lW}$} (A1_2);
    \path (A0_2) edge [->]node [auto] {$\scriptstyle{}$} (A1_2);
    \path (A1_0) edge [->,bend left=30]node [auto,swap] {$\scriptstyle{f'_\lW}$} (A1_2);
    \path (A1_2) edge [->]node [auto] {$\scriptstyle{}$} (A1_3);
    \path (A0_2) edge [->]node [auto] {$\scriptstyle{}$} (A0_3);
    \path (A0_0) edge [->]node [auto] {$\scriptstyle{}$} (A1_0);
  \end{tikzpicture}.
  \]
    \end{enumerate}
\end{lem}

\begin{proof} 
%

If we apply Lemma \ref{affine lift general} with $R=W$, $I=pW$ and $V=X$, so that $Z=\Spec k$ and $\lZ=\lW$, we obtain $(1)$, $(2)$ and $(4)$. 

Now consider the situation of $(3)$ and $(5)$ and set
$X_\lW=\Spec R$. We apply Lemma \ref{affine lift general} with
$I=pR$ and $V=Y$, so that $Z=X$ and $\lZ=X_\lW$. Lemma
\ref{affine lift general} (3) directly implies case $(5)$. From
Lemma \ref{affine lift general} $(2)$ we obtain a lift $V_\lZ\to
\lZ=X_\lW$ of $Y\to X$ and, using $(4)$, we find a
$\lW$-isomorphism $\phi$ from $Y_\lW \to \lW$ to $V_\lZ \to
X_\lW \to \lW$ which lifts $\id_Y\colon Y\to Y$: the composition
$Y_\lW \arrdi \phi V_\lZ \to X_\lW$ is the desired map $f_\lW$
in (3).
\end{proof}

\begin{prop}\label{Crystals abelian}
 If $X$ is a smooth and quasi-compact $k$-scheme then $\Crys(X/\lW)$ is a symmetric monoidal, abelian category with internal homomorphisms.
\end{prop}

\begin{proof}[Proof of Theorem \ref{isocrystannaka} and Proposition \ref{Crystals abelian}]
    Firstly note that the category $\Crys(X/\lW)$ is a symmetric monoidal additive $\lW$-linear category. It also admits cokernels as the pullback functor of 
    sheaves of modules is right exact and cokernels of maps of finitely presented
    modules are still finitely presented (\cite[\href{https://stacks.math.columbia.edu/tag/0519}{Tag 0519}]{stacks-project}).

    Now we consider the existence of kernels and the internal homomorphisms.
    Let $\{U_i\}_{i\in I}$ be a finite Zariski covering of $X$ such that each $U_i$ is 
    an affine non-empty scheme. Taking into account Lemma \ref{affine lift}, for each $U_i$ we can choose a smooth
    lift $\lU_i=\Spec A_i\to \lW$ of $U_i\to \Spec k$. Set $(U_i)_\lW$ for the spectrum of the $p$-adic completion of
    $A_i$. By Lemma \ref{subcategory of quasi-nilpotent} and
    Theorem \ref{cry equivalent to quasi-nil}  we see that each
    $\Crys(U_i/\lW)$ admits kernels and internal homomorphisms.

    It is straightforward that $\Crys(-/\lW)$ is a stack 
    on the small Zariski site of $X$. If $U_{ij}$ is a non-empty affine open inside $U_i\cap U_j$, then 
    by Lemma \ref{affine lift} (3) there is a flat $\lW$-lift
    $(U_{ij})_{\lW}\to (U_i)_{\lW}$ (note that this is not an
    open immersion!), whose flatness implies that kernels and internal homomorphisms are preserved at the level of topologically quasi-nilpotent connections
    by the pullback. We can glue kernels and internal homomorphisms in $\Crys(X/\lW)$ using the universal property defining them.

    Thus we can conclude that $\Crys(X/\lW)$ and, by Lemma \ref{localization abelian}, $I_{\rm{crys}}(X/\lW)$ are abelian categories, because the canonical map from the coimage to the image is an isomorphism (as it is an isomorphism
    when restricted to each $U_i$). Moreover by construction and
    again by Lemma \ref{localization abelian} restriction to an open
    is exact, tensorial and preserves internal homomorphisms for both $\Crys(-/\lW)$ and $I_{\rm{crys}}(-/\lW)$. This ends the proof of Proposition \ref{Crystals abelian}.
    
    We now deal with the proof of Theorem \ref{isocrystannaka}. In particular we assume that $X$ is connected. In particular $X$ and all $U_i$ are integral schemes. 
   Since we are in the
    situation of Remark \ref{1.3cartesian}, the category
    ${I_{\rm{crys}}(U_i/\lW)}$ is Tannakian by
    Theorem \ref{diff modules and connections are tannakian}.
    
    It is easy to check that $I_{\rm{crys}}(-/\lW)$ is a
    prestack in the small Zariski site of $X$, that is
    morphisms between isocrystals form a Zariski sheaf. In
    particular $I_{\rm{crys}}(X/\lW)$ is rigid because all
    $I_{\rm{crys}}(U_i/\lW)$'s are Tannakian.
    
    Next we will show that the ring of endomorphisms of the
    trivial object $\sO_{X/\lW}\otimes_W K\in\Icrys(X/\lW)$ is a
    field. Let $\phi$ be a non zero endomorphism of
    $\sO_{X/\lW}\otimes_W K$. We must show that $\phi$ is
    invertible. Since $\Icrys(-/\lW)$ is a prestack we must show
    that its restriction $\phi_i$ over $U_i$ is invertible. As
    $\Icrys(U_i/\lW)$ is Tannakian, it is enough to show that
    $\phi_i\neq 0$. By contradiction assume that $\phi_i=0$.
    The functor $\Icrys(U_j/\lW)\arr
    \Icrys(U_{ij}/\lW)$ is exact, $K$-linear and tensorial, so
    it is faithful by \cite[2.10]{Del90}. Since $(\phi_i)_{|U_{ij}}=0$, we have  $\phi_j=0$ for all $j$
    by the connectedness of $X$. But this would imply that $\phi=0$.
    
    Hence the endomorphisms of $\sO_{X/\lW}\otimes_W K$ form a field. Let's denote it by $L$. A fiber functor for $\Icrys(X/\lW)$ is obtained composing a fiber functor of $\Icrys(U_i/\lW)$ with the tensor exact functor $\Icrys(X/\lW)\arr \Icrys(U_i/\lW)$.
    
    In conclusion $I_{\rm{crys}}(X/\lW)$ is a Tannakian category over $L$ (see \cite[1.9]{Del90}).

 Let now $f\colon Y\arr X$ be a map as in the statement of Theorem \ref{isocrystannaka} and denote by $ f_{\rm{crys}}^* \colon \Icrys(X/\lW) \arr \Icrys(Y/\lW)$
 the  pullback. 
 We know that $f_{\rm{crys}}^*$ is a tensor functor and we
 must show that it is exact. 
 
 Let $U\subseteq X$ and $V\subseteq
 Y$ be non-empty affine open subsets such that $f(V)\subseteq
 U$. Let $\lf\colon V_{\lW}\arr U_{\lW}$ be a lift of $V\arr U$ as in Lemma \ref{affine lift}, (3)  and $v\colon
 \Spec \mathbb{K} \arr V_{\lW}\times_\lW K$ be a geometric point. Using
 Theorem \ref{cry equivalent to quasi-nil}  we have a commutative diagram
     \[
  \begin{tikzpicture}[xscale=4.9,yscale=-1.1]
    \node (A0_0) at (0, 0.5) {$\Icrys(X/\lW)$};
    \node (A0_1) at (1, 0.5) {$\Icrys(Y/\lW)$};
    \node (A2_0) at (0, 2) {$\Icrys(U/\lW)$};
    \node (A2_1) at (1, 2) {$\Icrys(V/\lW)$};
    \node (A3_0) at (0, 3) {$\QNCf(U/\lW)\otimes K$};
    \node (A3_1) at (1, 3) {$\QNCf(V/\lW)\otimes K$};
    \node (A3_2) at (2, 3) {$\Vect(\mathbb{K})$};
    \path (A2_1) edge [->]node [rotate=-90,above] {$\scriptstyle{\simeq}$} (A3_1);
    \path (A2_0) edge [->]node [rotate=-90,above] {$\scriptstyle{\simeq}$} (A3_0);
    \path (A3_0) edge [->,bend left=50]node [auto,swap] {$\scriptstyle{(\lf \circ v)^*}$} (A3_2);
    \path (A0_0) edge [->]node [auto] {$\scriptstyle{f_{\rm{crys}}^*}$} (A0_1);
    \path (A0_1) edge [->]node [auto] {$\scriptstyle{}$} (A2_1);
    \path (A3_0) edge [->]node [auto] {$\scriptstyle{\lf^*}$} (A3_1);
    \path (A0_1) edge [->]node [auto] {$\scriptstyle{}$} (A3_2);
    \path (A2_0) edge [->]node [auto] {$\scriptstyle{}$} (A2_1);
    \path (A0_0) edge [->]node [auto] {$\scriptstyle{}$} (A2_0);
    \path (A3_1) edge [->]node [auto] {$\scriptstyle{v^*}$} (A3_2);
  \end{tikzpicture}.
  \]
Notice that $v^*\colon \QNCf(V/\lW)\otimes K\arr \Vect(\mathbb{K})$ is the composition
\[
\QNCf(V/\lW)\otimes K \arr \Conn(V_{\lW}\times_\lW K/K,
\Omega_{V_\lW}\otimes_WK)\arr \Vect(\mathbb{K})
\]
and it is a fiber functor by construction (or we can check it directly because modules in the middle category are locally free). The same happens to $U$ and $(\lf\circ v)^*$. In particular those arrows and therefore also $\Icrys(X/\lW)\arr \Vect(\mathbb{K}) ,\Icrys(Y/\lW) \arr \Vect(\mathbb{K})$ are exact and faithful. From this it follows that $\Icrys(X/\lW)\arr\Icrys(Y/\lW)$ is exact.

Let's conclude computing $\Icrys(X/\lW)$ for $X=\Spec k$. We
have $X_{\lW}=\lW$ and, in particular, $\Omega_{X_{\lW}}=\Omega_{ \lW/\lW}=0$. In particular $\QNCf(X/\lW)=\Conn(X/\lW)$ is just the category of finitely generated $\lW$-modules. Tensoring by $K$ one exactly gets $\Vect(K)$. 
\end{proof}

\section{Base change theorems for crystalline cohomology}\label{section:basechange}
In this section we generalise in various ways the classical base change theorem for crystalline cohomology proven in \cite[V, Proposition 3.5.2]{Ber74}, \cite[Theorem 7.8]{BO78}. Let $k$ be
a perfect field of characteristic $p>0$, and let $W$ be the ring of
Witt vectors of $k$. Set $\lW\coloneq \Spec W$. We denote by $\gamma$ the canonical
$\mathrm{PD}$-structure on $pW$, $K$ the fraction field of
$W$. Set $W_n\coloneq W/p^{n}W$ and $\lW_n\coloneq\Spec W_n$. We denote by $\gamma_n$ the
induced $\mathrm{PD}$-structure on $pW_n$

\begin{conv}\label{base change convention} Let $\lS=(\lS, I, \gamma)$ be a   PD-scheme such that $\lS$ is a $\lW$-scheme and $p\in I$. Denote by $S$ the zero locus $V(I)$ of $I$ inside $\lS$, which is a $k$-scheme because $p\in I$. Let $X$ be an $S$-scheme and denote by $g\colon X\arr \lS$ the structure map. Consider a commutative diagram
  \[
\xymatrix{
    X'\ar[r]^-h\ar[d]_{g_0'}\ar@/_4pc/[dd]_-{g}&X\ar[d]^{g_0}\ar@/^4pc/[dd]^-{g}\\
    S'\ar[r]^-{h_0}\ar[d]&S\ar[d]\\
\lS'\ar[r]^-{u}& \lS
}
\]
where $\lS'=(\lS', I', \gamma')$ is a PD-scheme, $S'=V(I')$, $X'$
is a scheme, $u$ is a PD-morphism and the top square  is cartesian. We assume moreover that all schemes are quasi-compact and $g_0$ is smooth, quasi-compact and quasi-separated. We consider a crystal of finite presentation $E\in \Crys(X/\lS)$.
\end{conv}

We define
\[
    \Gamma((X/\lS)_{\rm{crys}}, - ) \colon \Mod(\sO_{X/\lS}) \arr \Mod(\Hl^0(\odi \lS))
\]
as the functor of global sections (\cite[p. 5.5]{BO78}).  It is easy to see that  \[
 \Gamma((X/\lS)_{\rm{crys}},E)=\varprojlim_n \Gamma((X/\lS_n)_{\rm{crys}}, E_{|(X/\lS_n)_{\rm{crys}}})\in \Mod(\widehat{\Hl^0(\odi\lS)})
 \]
where $E$ is a sheaf of $\sO_{X/\lS}$-modules on $(X/\lS)_{\rm{crys}}$,
$\lS_n\coloneq \lS\times_{\lW}\lW_n$ and
$\widehat{\Hl^0(\odi\lS)}$ is the $p$-adic completion of
$\Hl^0(\odi\lS)$.

There is a canonical projection from the crystalline ringed
topos to the Zariski ringed topos \cite[\href{https://stacks.math.columbia.edu/tag/07IL}{Tag
07IL}]{stacks-project}
$$u_{X/\lS}\colon((X/\lS)^{\sim}_{\rm{crys}},\sO_{X/\lS})\rightarrow
({X^{\sim}_{\rm{Zar}}, g^{-1}\sO_{\lS}})$$
 where $g^{-1}\sO_{\lS}$ is the
pullback of $\sO_\lS$ along $g$. Concretely, we have 
\begin{enumerate}
    \item For $F\in(X/\lS)^\sim_\crys$ and $j\colon
        U\hookrightarrow X$ an open,
        $$(u_{X/\lS_*}(F))(U)=\Gamma((U/\lS)_\crys,F);$$
    \item For $G\in X^\sim_{\rm{Zar}}$ and
        $(U,T,\delta)\in(X/\lS)_\crys$,
    $$({u_{X/\lS}}^*(G))(U,T,\delta)=G(U).$$
\end{enumerate}
By composition we get a morphism of topoi
$$g_{X/\lS}\coloneq g\circ u_{X/\lS}\colon((X/\lS)^{\sim}_{\rm{crys}},\sO_{X/\lS})\rightarrow
({\lS^{\sim}_{\rm{Zar}}, \sO_{\lS}}).$$
Notice
that $$\Gamma((X/\lS)_{\rm{crys}}, - )=\Gamma \circ
g_{X/\lS*}(-)$$ where $\Gamma\colon
\Mod(\sO_{\lS}) \arr \Mod(\Hl^0(\odi \lS))$ is the functor
of global sections.

\begin{lem}\label{various bounded complexes}
  Assume that $p$ is nilpotent in $\odi {\lS}$ and $\lS$ is separated.  Let
  $E\in \Crys(X/\lS)$. Then $\rr g_{X/\lS*}(E)$ is quasi-isomorphic to a
  bounded complex of quasi-coherent $\odi \lS$-modules. If $\lS$ is affine,
  then this is quasi-isomorphic to the complex of quasi-coherent
  $\sO_\lS$-modules associated with any complex of
  $\Hl^0(\sO_\lS)$-modules representing $\rr\Gamma((X/\lS)_{\rm{crys}},E)$.
 \end{lem}
\begin{proof}
 The complex $\rr g_{X/\lS*}(E)$ is cohomologically bounded and has
 quasi-coherent cohomology thanks to \cite[Theorem 7.6]{BO78}. By a standard
 argument it is quasi-isomorphic to a bounded complex of quasi-coherent sheaves
 $B^{\bullet}$ \cite[Corollary 5.5]{BokNee93}. If $\lS$ is affine the
 degenerate spectral sequence (see \cite[5.7.9]{Wei94})
 \[
 E_2^{pq}=\Hl^p((\rr^q\Gamma)(B^{\bullet})) \Rightarrow \rr^{p+q}\Gamma(B^{\bullet})
 \]
 tells us that $\Gamma(B^{\bullet})\simeq\rr\Gamma(B^{\bullet})\simeq\rr\Gamma((X/\lS)_{\rm{crys}},E)$ as desired.
\end{proof}

\begin{rmk}\leavevmode\label{adjgxsandgamma}
\begin{enumerate}[label=(\alph*)]
\item\label{adjgxs}
If $p$ is nilpotent in $\odi \lS$ we define a map 
 \begin{equation}\label{basechange}
\dl u^*\rr g_{X/\lS*}(E)\rightarrow \rr g'_{X'/\lS'*}h_{\rm{crys}}^*(E)
 \end{equation}
in $D(\lS_{\Zar}^\sim)$ as follows. Applying adjunction to the canonical map $\dl h^*_{\rm{crys}}(E)\rightarrow h^{*}_{\rm{crys}}(E)$
we obtain a map
$$E\rightarrow \rr h_{\rm{crys}*}( h^*_{\rm{crys}}(E)).$$
Applying $\rr g_{X/\lS*}$ and using $g_{X/\lS}\circ h_{\mathrm{crys}*}=u_*\circ g'_{X'/\lS'}$ (see \cite[\href{https://stacks.math.columbia.edu/tag/07MH}{Tag 07MH}]{stacks-project}) we get
$$\rr g_{X/\lS*}(E)\rightarrow \rr u_*\rr g'_{X'/\lS'*}(
h^*_{\rm{crys}}(E)).$$ 
The map \eqref{basechange} is obtained applying adjunction again, which can be
done because $\rr g_{X/\lS*}(E)$ is bounded above thanks to Lemma \ref{various bounded complexes}. 
\item\label{adjGamma}
If $\lS$ and $\lS'$ are affine, but $p$ is not necessarily nilpotent in
$\sO_\lS$, we can still define a map
\begin{equation}\label{basechange affine}
 \dl u^*\rr \Gamma((X/\lS)_{\rm{crys}},E) \rightarrow \rr \Gamma((X'/\lS')_{\rm{crys}},h_{\rm{crys}}^*E)
 \end{equation}
 in $D^-(\Hl^0(\odi{\lS'}))$. The construction is the same and it is possible
 since $\rr \Gamma((X/\lS)_{\rm{crys}},E)$ is bounded above as we will prove in
 Corollary \ref{RGamma Rlim}.

\end{enumerate}
\end{rmk}

\begin{defn}
 Let $\mathscr{A}$ be an abelian category, $p$ a given prime and $N\in \N$. A map of objects of $\mathscr{A}$ is  a $p^N$-isogeny if its kernel and its cokernel are killed by $p^N$, it is an isogeny if it is a $p^r$-isogeny for some $r\in \N$.
\end{defn}
\begin{defn}
 Let $\mathscr{C}$ be a $W$-linear category and $E\in \mathscr{C}$. Given $n\in
 \N$ we say that $E$ is $W_n$-flat if $p^n$ kills $E$ and, for all $0\leq j\leq
 n$, the quotient $E/p^jE$ exists and the map
 \[
 E/p^jE\arrdi{p^{n-j}} E
 \]
 is injective.
 We say that $E$ is $W$-flat or $p$-torsion free if $E\arrdi p
 E$ is injective in $\mathscr{C}$. 
\end{defn}
\begin{rmk}
 If $\mathscr{C}=\Mod(W)$ then the notion of flatness just introduced and the classical one agrees. This is an easy consequence of testing flatness on ideals.
\end{rmk}

\begin{lem}\label{restricting ptorsion free}
 In the hypothesis of \textbf{Setting} \ref{base change convention}, if $E\in
 \Crys(X/\lS)$ is $p$-torsion free, then $E_n = E_{|(X/\lS_n)_{\rm{crys}}}\in \Crys(X/\lS_n)$ is $W_n$-flat. 
 \end{lem}
\begin{proof}
Indeed,  since $\Crys(X/\lS)$ satisfies Zariski descent, we can assume
 that $X$ and $\lS$ are affine. We apply Theorem \ref{cry equivalent to
 quasi-nil} twice. The crystal $E$ corresponds to a module $M$
 with an integrable connection
 over $\Spec B$, where $B$ is a $p$-adically complete $W$-algebra and $\Spec B$ is a lift of $X$. The $B$-module $M$ is $p$-torsion free,  thus
 $W$-flat, so its restriction  $M_n\coloneq M\otimes_W W_n$ is
 $W_n$-flat. Therefore, the crystal $E_n$, which corresponds to $M_n$, is
 also $W_n$-flat. 
\end{proof}

\begin{rmk} Let $E\in \Crys(X/\lS)$ be $p$-torsion free.  It is
    not true that the map $E\xrightarrow{p}E$  is injective in
    the ringed topos $((X/\lS)^{\sim}_{\rm{crys}},\sO_{X/\lS})$. For example, we can look at the trivial
    crystal $\sO_{\Spec k/\lW}$ on $(\Spec k/\lW)_\crys$: the
    map  $W_1\xrightarrow{p} W_1$ at the thickening $\Spec
    k\hookrightarrow \lW_1$ is not injective.   
\end{rmk}

\begin{lem}\label{flat is torsion free}
    In the hypothesis of \textbf{Setting} \ref{base change convention}, if $\lS$
    is flat over $\lW$, and
    if $E\in\Crys(X/\lS)$ is a flat crystal \textup{\cite[p. 7.10]{BO78}}, then
$E$ is $p$-torsion free in  $\Crys(X/\lS)$. 
\end{lem}
\begin{proof}
Indeed, to see
this we may
assume $X$ and $\lS$ are affine. 
Then by Theorem \ref{cry equivalent to
 quasi-nil},  $E$ corresponds to a flat module $M$ equipped with
 an integrable connection over the
 flat $\lS$-lift $\Spec B$ of $X$, where $B$ is a $p$-adically complete $W$-algebra. Since $\lS$ is flat over $\lW$, $M$
 is $W$-flat, hence it is $p$-torsion free in $\QNCf(X/\lS)$. Thus $E$ is $p$-torsion
 free in $\Crys(X/\lS)$ as well by Theorem \ref{cry equivalent to
 quasi-nil}.
\end{proof}

\subsection{The case of a base killed by a power of $p$}\label{subsection:killed_by_p}

 The next theorem deals with the situation in Remark \ref{adjgxsandgamma} \ref{adjgxs} and the map in \eqref{basechange}.

\begin{thm}\label{Toninibasechange}
In the situation of \textbf{Setting} \ref{base change convention}, assume moreover that $p$ is nilpotent in $\odi {\lS}$. The following hold. 
\begin{enumerate}[label=\textup{(}\alph*\textup{)}]
\item \label{unifbounded}
There exists $r\in \N$, which depends only on
 $X\xrightarrow{g_0} S$, such that for all open  $U$ of $\lS$ and $i>r$ we have 
  $$
  \rr^i
  (g_{|g^{-1}(U)})_{g^{-1}(U)/U*}(E_{|(g^{-1}(U)/U)_{\rm{crys}}})=0.
  $$

\item \label{uflatEfree}
The map \eqref{basechange} is an isomorphism if $u$ is flat or
$E$ is a flat crystal \textup{\cite[p. 7.10]{BO78}}.
\item \label{EWnflatetc}
The map  \eqref{basechange} is an isomorphism if $E$ is
$W_n$-flat, $\lS$ is a flat $\lW_n$-scheme and if there exists a map
of schemes $u_0\colon \lZ\to\lW_n$ such that $u$ is the base change of
$u_0$ along $\lS\to\lW_n$.
\item \label{almostisoW}
    Suppose that $S$ is smooth of finite type over $k$. Let
    $E_{\lW}\in\Crys(X/\lW)$ and set $E=(E_{\lW})_{|(X/\lS)_\crys}$.
    Then there exists $N\colon \Z\arr \N$, independent of the closed immersion $S\hookrightarrow \lS$, such that the $i$-th cohomology of the map
    \eqref{basechange} is a $p^{N_i}$-isogeny in the ringed
    topos  $(\lS^{\prime\,\sim
    }_{\mathrm{Zar}},\sO_{\lS'})$.
 \end{enumerate}
\end{thm}

Before giving the proof of this theorem we prove some preliminary results.

 \begin{lem}\label{toninitoplemma}  
 Let $\pi\colon \mathscr{A}\arr \mathscr{B}$ be a left exact
 functor between abelian categories. Assume that $\mathscr{A}$
     has enough injectives and that there exists $n_0>0$ such that $\rr^n\pi=0$ for all $n\geq
         n_0$, so that, by
         \textup{\cite[\href{https://stacks.math.columbia.edu/tag/07K7}{Tag
         07K7}]{stacks-project}}, there is a functor $\rr \pi\colon D(\mathscr{A})\arr D(\mathscr{B})$.
  Let also $\alpha\colon C\arr D$ be a map in $D(\mathscr{A})$ and $N\colon
  \Z\arr \N$ a function such that $\Hl^i(\rr\pi(\alpha))$ is a $p^{N_i}$-isogeny and $N_i=0$ for $i\gg 0$. 
 
 Then there exists $N'\colon \Z\arr \N$, which depends only on
 $N$ and $n_0$, such that $\Hl^i(\rr\pi(\alpha))$ is a
 $p^{N_i'}$-isogeny and $N'_i=0$ for $i\gg 0$.
\end{lem}
\begin{proof}
Applying $\rr \pi$ to the exact triangle of the cone of $\alpha$ and taking cohomology we get a long exact sequence 
\[
    \cdots \rightarrow \rr^i \pi C \rightarrow	 \rr^i\pi
    D\rightarrow \rr^{i}\pi (\mathrm{Cone}(\alpha))\rightarrow
    \rr^{i+1}\pi C \rightarrow \cdots.
\]
From this we are reduced to show that if $G \in D(\mathscr{A})$ 
satisfies that $\Hl^i(G)$ is killed by $p^{N_i}$, 
then we can find $N'$ as in the statement such that $\rr^i \pi
G$ is killed by $p^{N_i'}$ and $N_i'=0$ for all $i\gg 0$.

We consider the truncation
\[
    \tau_{\geq n}(G)\coloneqq\cdots\rightarrow 0\rightarrow
    (G^n/\im(d^{n-1}))\xrightarrow{d^n}
    G^{n+1}\xrightarrow{d^{n+1}} G^{n+2}\rightarrow \cdots.
\]
By \cite[\href{https://stacks.math.columbia.edu/tag/08J5}{Tag
08J5}]{stacks-project}, we have an exact triangle
\[
\Hl^n(G)[-n]\arr \tau_{\geq n}(G)\rightarrow \tau_{\geq n+1}(G)\arr \Hl^n(G)[-n+1]
\]
hence the exact triangle
\[
\rr\pi(\Hl^n(G)[-n]) \arr \rr\pi(\tau_{\geq n}(G))\rightarrow \rr\pi(\tau_{\geq
n+1}(G))\arr \rr\pi(\Hl^n(G)[-n+1]).
\]
We show that there exists $f\colon \Z\to \N$ such that the multiplication by $p^{f_n}$ induces $0$ on all cohomologies of $\rr\pi(\tau_{\geq n}(G))$ and $f_n=0$ for $n\gg 0$.

For $n\in \N$ satisfying $N_m=0$ for $m\geq n$ we can set $f_n=0$. Indeed in this case $\tau_{\geq n}G$ (and therefore also
$\rr\pi(\tau_{\geq n}G)$) is acyclic by assumption.

Moreover $\Hl^n(G)$ and, by linearity, all $\rr\pi(\Hl^n(G)[u])$
($u\in \Z$) are killed by $p^{N_n}$ in the derived category.
We can therefore define $f\colon \Z\to \N$ working by reverse induction on $\Z$.

Next we show that
\[
\rr^i\pi(G)\arr \rr^i\pi(\tau_{\geq n}G)
\]
is an isomorphism for $n< i-n_0$ so that the function $N'_i=f_{i-n_0-1}$ satisfies the requests in the statement.

By \cite[\href{https://stacks.math.columbia.edu/tag/07K7}{Tag 07K7}]{stacks-project} we can assume that $G$ is made by right acyclic objects for $\pi$. For all $n\in \N$ we have an exact sequence of complexes
\[
0\arr \sigma_{\geq n+1}G\arr \tau_{\geq n}G \arr (G_n/\Imm(d_{n-1}))[-n]\arr 0
\]
where $\sigma_{\geq n+1}G$ denotes  truncation. Since
$\rr^q\pi=0$ for $q\geq n_0$, we can conclude that
$\rr^i\pi(\sigma_{\geq n+1}G)\arr \rr^i\pi(\tau_{\geq n}G)$ is
an isomorphism for $n\leq i-n_0$. Since $\rr\pi(G)=\pi(G)$ and $\rr\pi(\sigma_{\geq n+1}G)=\pi(\sigma_{\geq n+1}G)$ we can also conclude that
\[
\rr^i\pi(\sigma_{\geq n+1}G)\arr \rr^i\pi(G)
\]
is an isomorphism for $n+1  <i$. 
\end{proof}
 
 \begin{lem}\label{almost degenerate spectral sequence}
  Let $\mathscr{A}$ be an abelian category, $l,N\in \N$ and
  \[
  E^{uv}_2\Rightarrow H^{u+v}
  \]
  be a convergent spectral sequence in $\mathscr{A}$.
  
  If $E^{uv}_2=0$ for $v>0$ or $u<0$ or $u>l$, then there is an associated map 
   \[
  \omega_n \colon H^n \arr E^{n0}_\infty \arr E^{n0}_2
  \]
  and, if $p^N$ kills all $E^{uv}_2$ for $v\neq 0$, this map is
  a $p^{N(l+1)}$-isogeny.
  
   If $E^{uv}_2=0$ for $u>0$ or $v<0$ or $v>l$, then there is an associated map 
   \[
  \omega_n \colon E_2^{0n} \arr E_\infty^{0n} \arr H^n
  \]
  and, if $p^N$ kills all $E^{uv}_2$ for $u\neq 0$, this map is
  a $p^{N(l+1)}$-isogeny.
 \end{lem}
 \begin{proof}
 We consider only the first case because the second one is analogous.
  By convergence there is a filtration
\[
0=F^tH^n \subseteq\cdots\subseteq F^{u+1}H^n\subseteq F^uH^n\subseteq\cdots\subseteq F^sH^n=H^n
\]
for some $s<t$ such that
\[
E^{u,n-u}_\infty\simeq (F^uH^n)/(F^{u+1}H^n).
\]
The vanishing in the hypothesis tells us that
$F^uH^n=F^{u+1}H^n$ if $u<0$ or $u>l$ or $n>u$. Thus we
can choose $t=l+1$ and $s=\textrm{max}(0, n)$ in the above filtration. In particular, $F^nH^n=H^n$ for all $n$.

Since $E^{uv}_2=0$ for $v>0$ all differentials landing in $(u,0)$ are zero in all pages. It follows that $E^{u0}_\infty\subseteq E^{u0}_2$. Moreover there is a map
\[
\omega_n\colon H^n=F^nH^n\arr (F^nH^n)/(F^{n+1}H^n) \simeq
E^{n0}_\infty \hookrightarrow E^{n0}_2.
\]

Assume now that $p^N$ kills all the modules $E^{uv}_2$ for $v\neq 0$. It
follows that $p^N$ kills all modules $E^{uv}_r$ for $v\neq 0$ and $r\geq 2$. In
particular $E^{n0}_{r+1}$ is the kernel of a map from $E^{n0}_r$ to an object killed by $p^N$.
Moreover the differentials at page $l+1$ must be $0$, so that $E^{uv}_{l+1}=E^{uv}_\infty$. From this it follows that 
$$
\Coker \omega_n = E^{n0}_2/E^{n0}_\infty
$$ is killed by $p^{N(l+1)}$.

It remains to look at $\Ker(\omega_n) = F^{n+1}H^n$. But this object has a filtration of lenght $l$ of subobjects whose partial cokernels are killed by $p^N$. It follows that it must be killed by $p^{N(l+1)}$.
 \end{proof}
 
 \begin{lem}\label{projective p inverted}
  Let $B$ be a smooth $W$-algebra, let $\hat{B}$ the $p$-adic
  completion of $B$ and let $\Omega_{\hat{B}}$ be the $p$-adic
  completion of the module of algebraic differentials
  $\Omega_{\hat{B}/W}$ (As in Remark \ref{omegasurjective}, $\Omega_{\hat{B}}$ is a quotient of $\Omega_{\hat{B}/W}$ ). Let $(M,\nabla)\in
  \Conn(\hat{B}/W,\Omega_{\hat{B}})$. Then there exist
  $l,a,b\in \N$ and maps of $\hat{B}$-modules $\alpha\colon \hat B^l\arr M$, $\beta\colon M\to \hat B^l$ satisfing $p^a(\alpha\beta-p^b\id_{M})=0$. In particular if $F\colon \Mod(\hat B)\arr \mathscr{C}$ is any linear functor with values in a linear category and $F(\hat B)=0$ then $p^{a+b}$ kills $F(M)$.
 \end{lem}
 \begin{proof}
 The last claim follows by linearity applying the functor $F$ to the given expression and using that $F(\hat B)$ and therefore $F(\beta)$ are zero.
 
Applying Proposition \ref{locally simple implies locally free}, Lemma
\ref{connection as stratification} and Lemma \ref{connections in char zero are
locally simple} we can conclude that $M[1/p]$ is a finitely generated projective $\hat B[1/p]$-module. In particular there exist maps $\alpha\colon \hat B[1/p]^l\arr M[1/p]$ and $\beta\colon M[1/p]\arr \hat{B}[1/p]^l$ such that $\alpha\beta=\id$. Multiplying $\alpha$ and $\beta$ by a power of $p$ we can find $b\in \N$ and  $\alpha\colon \hat{B}^l\arr M$ and $\beta\colon M\arr \hat{B}^l$ such that $\alpha\beta=p^b\id$ in $ M[1/p]$. In particular there also exists $a\in \N$ such that $p^a(\alpha\beta-p^b\id_{M})=0$ as required.
\end{proof}


 \begin{proof}[Proof of Theorem \ref{Toninibasechange}]
We follow the proof of \cite[Theorem 7.8]{BO78}, in particular the proofs of \ref{unifbounded} and \ref{uflatEfree} are essentially the same  as the one given in the above reference.

We may assume that ${\lS}'$ and ${\lS}$ are affine.
We want to reduce to the case where $X$ is also affine by using cohomological descent as in \cite[Proposition 3.5.2]{Ber74} and \cite[Theorem 7.8]{BO78}. If $U\subseteq \lS$ is any open
subset then we have $$\rr
(g|_{g^{-1}(U)})_{g^{-1}(U)/U_*}(E|_{(g^{-1}(U)/U)_\crys})=\rr
g_{X/\lS_*}(E)|_U.$$  
Thus in \ref{unifbounded} we may assume $U=\lS$.

We take a finite affine covering $\{U_i\}_{i=0,\cdots,n}$  of
$X$.  From the covering we obtain the topos
$(X^{\bullet}/\lS)^{\sim}_{\rm{crys}}$ as in \cite[p.
335, p. 344]{Ber74}, and the morphism of topoi 
\[
\pi\colon(X^{\bullet}/\lS)^{\sim}_{\rm{crys}}\rightarrow
(X/\lS)^{\sim}_{\rm{crys}}.
\] 
Similarly, we have the topos $(X^{\prime\bullet}/\lS')^{\sim}_{\rm{crys}}$ and
the corresponding morphism of topoi $\pi':(X^{\prime\bullet}/\lS')^{\sim}_{\rm{crys}}\rightarrow (X'/\lS')^{\sim}_{\rm{crys}}$. Thus we have a diagram of topoi

\begin{equation*}
\xymatrix{
(X^{\prime\bullet}/\lS')^{\sim}_{\rm{crys}}\ar[r]^{h^{\bullet}_{\rm{crys}}}\ar[d]^{\pi'}\ar@/_3pc/[dd]_{g'_{X^{\prime \bullet}/\lS'}}&(X^{\bullet}/\lS)^{\sim}_{\rm{crys}}\ar[d]^{\pi}\ar@/^3pc/[dd]^{g_{X^\bullet/\lS}}\\
(X^{\prime}/\lS')^{\sim}_{\rm{crys}}\ar[r]^{h_{\rm{crys}}}\ar[d]^{g^{\prime}_{X'/\lS'}}&(X/\lS)^{\sim}_{\rm{crys}}\ar[d]^{g_{X/\lS}}\\
\lS^{\prime \sim}_{\rm{Zar}}\ar[r]^u&\lS^{\sim}_{\rm{Zar}.}
}
\end{equation*}
Then cohomological descent implies that there are canonical isomorphisms \cite[V, Proposition 3.4.8]{Ber74}
\[
E\xrightarrow{\cong} \rr\pi_*(\pi^*E)
\hspace{30pt}\text{and}\hspace{30pt}
h_{\rm{crys}}^*E\xrightarrow{\cong}
\rr\pi_*^{\prime}(\pi^{\prime*}h_{\rm{crys}}^*(E)).
\]
Applying $\rr g_{X/\lS*}$ to the first above isomorphism, $\rr g'_{X'/\lS'*}$ to the second, we obtain the following commutative  diagram 
\begin{equation}\label{twoiso}
\xymatrix{
    \dl u^*\rr g_{X/\lS*}(E)\ar[r]^(.45){\cong}\ar[d]&\dl u^*\rr g_{X^{\bullet}/\lS*}(\pi^*(E))\ar[d]\\\
\rr g'_{X'/\lS'*}h^*_{\rm{crys}}(E)\ar[r]^(0.45)\cong&\rr
g^{\prime}_{X^{\prime
\bullet}/\lS'*}h^{\bullet*}_{\rm{crys}}(\pi^*(E)).
}
\end{equation}
The vertical map on the right is obtained via adjunctions as in
\textbf{Setting} \ref{base change convention}, using that $\rr g_{X^{\bullet}/\lS}(\pi^*(E))$ is bounded being isomorphic to $\rr g_{X/\lS_*}(\rr\pi_*(\pi^*E))=\rr g_{X/\lS_*}(E)$ which is bounded by \cite[Theorem 7.6]{BO78}.
This means that we can work with $X^{\bullet}$ and $X^{\prime\bullet}$, instead of $X$ and $X'$ respectively.

Now let $\Delta$ be the opposite category of the category whose objects are subsets of $I\coloneqq \{0,1,2,\hdots,n\}$ and whose morphisms are the inclusions of subsets. As in \cite[V, 3.4.3]{Ber74} we obtain the commutative diagram 
\begin{equation*}
\xymatrix{
(X^{\prime\bullet}/\lS')^{\sim}_{\rm{crys}}\ar[r]^{h^{\bullet}_{\rm{crys}}}\ar[d]^{g^{\prime\bullet}_{X^{\prime\bullet }/\lS'}}\ar@/_3pc/[dd]_{g^{\prime}_{X^{\prime\bullet }/\lS'}}&(X^{\bullet}/\lS)^{\sim}_{\rm{crys}}\ar[d]^{g^{\bullet}_{X^{\bullet }/\lS}}\ar@/^3pc/[dd]^{g_{X^{\bullet}/\lS}}\\
(\lS^{\prime\sim}_{\rm{Zar}})^{\Delta}\ar[r]^{u^{\bullet}}\ar[d]^{\omega'}&(\lS_{\rm{Zar}}^{\sim})^{\Delta}\ar[d]^{\omega}\\
\lS^{\prime \sim}_{\rm{Zar}}\ar[r]^u&\lS_{\rm{Zar}}^{\sim}.
}
\end{equation*}
We know that $\rr g^{\bullet}_{X^{\bullet }/\lS}(\pi^*(E))$ has
bounded cohomologies by \cite[pp.~340, 320]{Ber74}. Then  by \cite[V. 3.4.9]{Ber74}, one has the isomorphism 
\begin{equation}\label{map}
\dl u^*(\rr\omega_*(\rr g^{\bullet}_{X^{\bullet
}/\lS_*}(\pi^*(E))))\xrightarrow{\cong} \rr\omega^{\prime}_*(\dl
u^{\bullet*}(\rr g^{\bullet}_{X^{\bullet }/\lS_*}(\pi^*(E)))).
\end{equation}
Note that by \cite[Prop. V. 3.4.9, i), p. 340]{Ber74} we have
$\rr^{i}\omega_*(-)=\rr^{i}\omega_*'(-)=0$ for all $i\geq n+1$
    or $i<0$
, so by
\cite[\href{https://stacks.math.columbia.edu/tag/07K7}{Tag
07K7}]{stacks-project} $\rr\omega$ and $\rr\omega'$ make sense. 
The right vertical arrow in \eqref{twoiso} is the composition of (\ref{map}) with  the map obtained by applying $\rr\omega'_*(-)$ to
\begin{equation}\label{reduction to affine}
    \dl u^{\bullet*}\rr
    g^{\bullet}_{X^{\bullet}/\lS_*}(\pi^*(E))\rightarrow  \rr
    g^{\prime
    \bullet}_{X^{\prime\bullet}/\lS_*'}(h^{\bullet*}_{\rm{crys}}(\pi^*(E))).
\end{equation}

Therefore, 
in \ref{unifbounded}, 
\ref{uflatEfree} and \ref{EWnflatetc}
  we can replace $\lS^{\sim}_{\rm{Zar}}$ and $\lS^{\prime
  \sim}_{\rm{Zar}}$  by $(\lS^{\sim}_{\rm{Zar}})^{\Delta}$ and
  $(\lS^{\prime \sim}_{\rm{Zar}})^{\Delta}$ respectively. When
  \ref{unifbounded} is proved, we can conclude that both  $\dl
  u^{\bullet*}\rr
  g^{\bullet}_{X^{\bullet}/\lS_*}(\pi^*(E))$ and $ \rr
  g^{\prime
  \bullet}_{X^{\prime\bullet}/\lS_*'}(h^{\bullet*}_{\rm{crys}}(\pi^*(E)))$
  have cohomologies bounded from above with a bound $i_0$ depending
  only on  $g_0$. Thus in \ref{almostisoW} we can also
  replace $\lS^{\sim}_{\rm{Zar}}$ and $\lS^{\prime
  \sim}_{\rm{Zar}}$  by $(\lS^{\sim}_{\rm{Zar}})^{\Delta}$ and
  $(\lS^{\prime \sim}_{\rm{Zar}})^{\Delta}$ respectively, because  we can reset the
  $N$ obtained for  $(\lS^{\sim}_{\rm{Zar}})^{\Delta}$ and
  $(\lS^{\prime \sim}_{\rm{Zar}})^{\Delta}$ to 
  \begin{equation*}
N_i' \coloneq \begin{cases}
    {N}_i &\text{if $i\leq i_0$}\\
    0 &\text{if $i > i_0$}
\end{cases}
\end{equation*} so that the conditions of Lemma
\ref{toninitoplemma} are satisfied.

By
\cite[Prop.~V.3.4.4]{Ber74} and \cite[Prop.~V.3.4.5]{Ber74} we see that $\dl u^{\bullet*}$ and $\rr g^\bullet_{X^\bullet/\lS_*}$ are computed componentwise.
An intersection of open affine subsets of $X$ may not be affine, but it is separated. Thus one can first reduce the problem to the case when $X$ is separated and, after, to the case when $X$ is affine.

Now let $\lS=\Spec A$ and $\lS'=\Spec A'$.
Since $g_0\colon X\to S$ is smooth and $X,S$ are affine, there
is a smooth affine lift $\tilde{g}_0\colon\Spec B=\lX\to \lS$ by \cite[\href{https://stacks.math.columbia.edu/tag/07M8}{Tag 07M8}]{stacks-project}, and by pulling back along $u\colon\lS'\to\lS$ we get a lift 
of $g_0'$ to $\tilde{g}_0'\colon \Spec B'=\lX'\to\lS'$. The comparison theorem (e.g. \cite[\href{https://stacks.math.columbia.edu/tag/07LG}{Tag 07LG}]{stacks-project}) tells us that there is a commutative diagram
\begin{equation}\label{de rham versus crystalline}
    \xymatrix{ \dl u^*\rr g_{X/\lS_*}(E)\ar[r]\ar[d]_-\cong& \rr g_{X'/\lS'_*}h_{\crys}^*(E)\ar[d]^-\cong\\
        \dl u^*\tilde{g}_{0_*}(\mathcal{M} \otimes_{ \sO_{\lX}}\Omega^\bullet_{\lX/\lS})\ar[r]^-\phi&\tilde{g}_{0_*}'(
    \mathcal{M}'\otimes_{ \sO_{\lX'}}\Omega^\bullet_{\lX'/\lS'})}
\end{equation}
where $\mathcal{M}\otimes_{\sO_{\lX}}\Omega^\bullet_{\lX/\lS}$ and
$\mathcal{M}' \otimes_{ \sO_{\lX'}}\Omega^\bullet_{\lX'/\lS'}$ are the de
Rham complex associated to the topologically quasi-nilpotent connections
corresponding to the crystal $E$ and $h_{\crys}^*E$ respectively via the map in
Theorem \ref{cry equivalent to quasi-nil}.

\begin{proof}[Proof of \ref{unifbounded}]
We see from the comparison theorem  \cite[\href{https://stacks.math.columbia.edu/tag/07LG}{Tag 07LG}]{stacks-project} that  $\rr g_{X/\lS_*}(E)$ has
bounded cohomologies whose bound depends only on the relative
dimension of $g_0\colon X\to S$,  so  the proof of \ref{unifbounded} is finished. 
\end{proof}


\begin{proof}[Proof of \ref{uflatEfree}]
Replacing the affine schemes $\lX,\lS,\lX',\lS'$ by the
 rings $B,A,B',A'$ and the quasi-coherent sheaves
$\mathcal{M},\mathcal{M}'$ by modules $M,M'$ respectively we
obtain the map
$$\phi\colon \dl
u^*(M\otimes_{B} \Omega^\bullet_{B/A})\arr
M'\otimes_{B'}\Omega^\bullet_{B'/A'}$$ which is the ring version of
the map
$\phi$ in \eqref{de rham versus crystalline} (which we still call $\phi$).
The functoriality in Theorem \ref{cry equivalent to quasi-nil} tells us that $M'=M\otimes_{ B}  B'$. Since  $\Omega^\bullet_{{B}'/A'}=\Omega^\bullet_{B/A}\otimes_{A} A'$
we see that the target of $\phi$ is just $u^*(M \otimes_{B}\Omega^\bullet_{B/A})$. If we denote by $Q^{\bullet}$ the bounded complex of $A$-modules $M \otimes_{B} \Omega^{\bullet}_{B/A}$, it follows that the map $\phi$ we are considering is the canonical map  
\begin{equation}\label{LuQtouQ}
\dl u^*(Q^\bullet)\arr u^*Q^\bullet.
\end{equation}
When $u$ is flat the map \eqref{LuQtouQ} is a
quasi-isomorphism. The same holds if $E$ is flat because in this case $Q^\bullet$ is a complex
of flat $A$-modules.
\end{proof}

\begin{proof}[Proof of \ref{EWnflatetc}] We proceed as in
    \ref{uflatEfree} and  get the map \eqref{LuQtouQ}. Assume
    that $E$ is $W_n$-flat and that there is a map of rings
    $W_n\to R$  such that $A'=A\otimes_{W_n} R$ as an
    $A$-algebra, then the module $M$ is $W_n$-flat in
    $\Mod(B)$ and therefore it is flat as $W_n$-module. Therefore the complex $Q^\bullet$ is a
    complex of flat $W_n$-modules. Using the
    flatness of $A$ and $Q^i$ over $W_n$ one can easily check that
\[
\dl u^*(Q^\bullet)=Q^\bullet\overset{\dl}{\otimes}_A A' \simeq
Q^\bullet \overset{\dl}{\otimes}_A (A\otimes_{W_n} R)
\simeq Q^\bullet \overset{\dl}{\otimes}_{W_n} R \simeq Q^\bullet
\otimes_{W_n} R.
\]
Thus \eqref{LuQtouQ} is an isomorphism.
\end{proof}
    
\begin{proof}[Proof of \ref{almostisoW}]
We proceed as in
    \ref{uflatEfree} and  get the map \eqref{LuQtouQ}. We consider the
    converging cohomological spectral sequences \cite[Proposition 5.7.6, with
    the convention on Dual Definition 5.2.3]{Wei94}
\[
E_2^{xy}=\Hl^x(\dl^yu^*(Q^{\bullet}))\Rightarrow \dl^{x+y} u^*(Q^{\bullet})=H^{x+y}
\] where $\dl^yu^*(Q^{\bullet})$ is the complex obtained by
applying $\dl^yu^*$ on each terms of $Q^\bullet$.
This sequence is obtained from the double complex made by the projective resolutions of the modules in $Q^\bullet$. 
It is a fourth quadrant spectral sequence, \emph{i.e.} $E_2^{xy}=0$ when $y>0$ or $x<0$ or $x>m$ (where $m$ is the relative dimension of $g_0$).

By Lemma \ref{almost degenerate spectral sequence}, from the spectral sequence, we obtain a map
\begin{equation}\label{edgemapuQ}
H^i=\dl^iu^*(Q^\bullet) \arr \Hl^i(u^*Q^\bullet)=E^{i0}_2
\end{equation}
which is the $i$-th cohomology of the map we are considering.

By Lemma \ref{affine lift} we have the following diagram
\[
\begin{tikzpicture}[xscale=2.9,yscale=-1.2]
    \node (A0_0) at (0, 0) {$X$};
    \node (A0_1) at (1, 0) {$X_{\lW}\coloneq \Spec \widetilde{B}$};
    \node (A1_0) at (0, 1) {$S$};
    \node (A1_1) at (1, 1) {$S_{\lW}\coloneq\Spec \widetilde{A}$};
    \node (A2_0) at (0, 2) {$\Spec k$};
    \node (A2_1) at (1, 2) {${\lW}$};

\path (A0_0) edge [->] node[auto] {$\scriptstyle{}$} (A0_1);
\path (A1_0) edge [->] node[auto] {$\scriptstyle{}$} (A1_1);
\path (A2_0) edge [->] node[auto] {$\scriptstyle{}$} (A2_1);
\path (A1_0) edge [->] node[auto] {$\scriptstyle{}$} (A2_0);
\path (A1_1) edge [->] node[auto] {$\scriptstyle{}$} (A2_1);
        \path (A0_1) edge [->] node[below,midway,scale=.9]
            {$\scriptstyle{}$} (A1_1);
    \path (A0_0) edge [->] node[auto] {$\scriptstyle{}$} (A1_0);
\end{tikzpicture}
\]
where $\widetilde{A}$ is the $p$-adically complete flat lift of the smooth
$k$-algebra $A/I$ to $W$, and $\widetilde{B}$ is the $p$-adically
complete flat
lift of the smooth $A/I$-algebra $B/I$ to $\widetilde{A}$. Since
$\widetilde{A}$ is $p$-adically formally smooth over $W$ and
$A\to A/I$ is a quotient of $p$-adically
discrete $W$-algebras defined by a nil ideal, we can choose a $W$-map $\widetilde{A}\to A$.
In the same way, we can choose an $\widetilde{A}$-map $\widetilde{B}\to B$.

If $(\hat M,\hat \nabla)\in \QNCf( X/\lW)$ is the
quasi-nilpotent connection corresponding to $E_{\lW}\in
\Crys(X/\lW)$ via Theorem \ref{cry equivalent to quasi-nil}, then $\hat
M\otimes_{\widetilde B} B\simeq M$ by the crystalline nature of $E_{\lW}$. Applying Lemma 
\ref{projective p inverted} to $\hat M$ and the functors
\[
    F_{yt}(-)\coloneq \dl^y u^*(- \otimes_{\widetilde B} \Omega^t_{B/A})
        \hspace{50pt} (y\in\Z_{\neq 0}, t\in \N)
\]
we find $N'\in \N$  such that $p^{N'}$ kills $F_{yt}(\hat{M})$ hence
also
$\Hl^x(F_{y\bullet}(\hat M))=E_2^{xy}$ ($y\neq 0$). 
Notice that $N'\in \N$ depends only on the $\widetilde{B}$-module
$\hat{M}$ and the
$\widetilde{B}$-module $\hat{M}$ depends only
on $E_{\lW}\in \Crys(X/\lW)$. 

  Set $N_i\coloneq (m+1)N'$ for all $i\in\Z$, where $m$ is the relative dimension
of $g_0$. By Lemma  \ref{almost degenerate spectral
  sequence},
 the $i$-th cohomology of
\eqref{basechange} is a $p^{N_i}$-isogeny.
\end{proof}
The proof of the theorem is done.
\end{proof}

\begin{rmk}
Note that in the proof of Theorem \ref{Toninibasechange}
(a), (d), the bound $r$ and the function $N\colon \Z\to \N$
depend not only on the relative dimension of $g_0$, but also on
the number of  opens in the affine covering
$\{U_i\}_{i=0,\cdots,n}$ of $X$ and the affine coverings of the arbitrary
intersections of $\{U_i\}_{i=0,\cdots,n}$. Indeed this was used during the
reduction of $X$ to the affine case (see the two paragraphs after \eqref{reduction to
affine}). Since this is a choice on  $X$ which is part of the
map $g_0:X\to S$, we didn't specify it.
\end{rmk}

\subsection{The case of an affine base}\label{subsection:affine}
In this section we treat the case in which the base $\lS$ is affine. The first result is a corollary of the base change theorem proven in the previous subsection (Theorem \ref{Toninibasechange}). 
\begin{cor}\label{RGamma Rlim} 
    In the situation of \textbf{Setting} \ref{base change convention} assume that $\lS=\Spec A$ is affine  and set $A_n=A/p^n$, $\lS_n=\Spec A_n$ and $E_n=E_{|(X/\lS_n)_{\rm{crys}}}$. Then the following hold. 
\begin{enumerate}[label=\textup{(}\alph*\textup{)}]
 \item There exists $r\in\N$, which depends only on
     $X\xrightarrow{g_0}S$ such that for all  $i\geq r$ we have
     $\rr^i\Gamma((X/\lS_n)_\crys, E_n)=0$.
     Moreover, we have
     \begin{equation}\label{Rlim}
 \rr\Gamma((X/\lS)_{\rm{crys}},E) \simeq \rr\varprojlim_n \rr\Gamma((X/\lS_n)_{\rm{crys}},E_n)
 \end{equation}
 is  quasi-isomorphic to a bounded complex.
 \item
 If $\lS$ is flat over $\lW$ and $E$ is $p$-torsion free then the
 system $\{\rr\Gamma((X/\lS_n)_{\rm{crys}},E_n)\}$ is
 quasi-consistent in the sense of \textup{\cite[B.4]{BO78}} and, if
 moreover $\lS$ is Noetherian, then $\rr\Gamma((X/\lS)_{\rm{crys}},E)\overset{\dl}{\otimes}_A  A_n \simeq \rr\Gamma((X/\lS_n)_{\rm{crys}},E_n)$.
 \item
If $\lS$ is flat over $\lW$ and Noetherian, $E$ is $p$-torsion free, $X/S$ is proper and $S=S_1$ $(I=(p))$, then $ \rr\Gamma((X/\lS)_{\rm{crys}},E)$ is quasi-isomorphic to a bounded complex of finitely generated $\hat A$-modules, where $\hat A$ is the $p$-adic completion of $A$, and
 \[
  \Hl^i((X/\lS)_{\rm{crys}},E) \simeq \varprojlim_n \Hl^i((X/\lS_n)_{\rm{crys}},E_n).
 \]
 Moreover, the projective system on the right hand side satisfies
 the Mittag-Leffler condition,  and is made by finitely generated $\hat A$-modules.
 \end{enumerate}
\end{cor}
\begin{rmk}
The proof of Corollary \ref{RGamma Rlim} is the same as the proof of \cite[Proposition 5.3 1)]{ES16} and \cite[\textbf{claim} in pp.~10--11]{Shi07I} .  
\end{rmk}
\begin{proof}[Proof of Corollary \ref{RGamma Rlim}] 
Firstly, notice that we can replace $A$ by its $p$-adic
completion thanks to
\cite[\href{https://stacks.math.columbia.edu/tag/05GG}{05GG}]{stacks-project}.
\begin{enumerate}[label=\textup{(}\alph*\textup{)}]
\item
    The isomorphism \eqref{Rlim} follows from \cite[\href{https://stacks.math.columbia.edu/tag/07MV}{Tag 07MV}]{stacks-project}. By Theorem \ref{Toninibasechange} \ref{unifbounded} and \cite[Remark B.1.6]{BO78} we also get the boundness.
\item
 The quasi-consistency follows from Lemma \ref{restricting ptorsion free} and Theorem \ref{Toninibasechange} \ref{EWnflatetc} because the maps $\lS_{n-1}\arr \lS_n$ are  base changes of the maps $\Spec W_{n-1}\arr\Spec W_n$.
From the quasi-consistency and \cite[Proposition B.5, 3)]{BO78} we obtain the last isomorphism.
 \item
Assume that $\lS$ flat over $\lW$, $E$ is $p$-torsion free, $X/S$ is proper and $S=S_1$ ( $I=(p)$).
 Since $\rr\Gamma((X/\lS_1)_{\rm{crys}},E_1)$ has finitely generated cohomologies and all the $\rr\Gamma((X/\lS_n)_{\rm{crys}},E_n)$ are uniformly cohomologically bounded thanks to \cite[p.7.7]{BO78}, the result follows from \cite[Lemma B.6 and Proposition B.7]{BO78}. Here we use that a bounded  complex with finitely generated cohomology is quasi-isomorphic to a bounded complex of finitely generated modules.\qedhere
 \end{enumerate}
\end{proof}

Always in \textbf{Setting} \ref{base change convention}, we consider now the situation in Remark  \ref{adjgxsandgamma} \ref{adjGamma}. We analyse under which condition the map in \eqref{basechange affine} is an isomorphism (or an isogeny). 

\begin{thm}
\label{Toninibasechange Gamma} Let  the notation and hypothesis
be as in \textbf{Setting} \ref{base change convention}. Assume moreover
that $\lS$ is Noetherian and $\lW$-flat. Let $\lS'\coloneq \Spec
A'$, $\lS\coloneq \Spec A$, where $A$ and $A'$ are $p$-adically
complete rings. Suppose that one of the following is true: $p$ is
nilpotent in $\odi{\lS'}$ or $X/S$ is proper and $S=\Spec A/p$
\emph{(i.e.} $I=(p))$. 
Then the following hold. 
\begin{enumerate}[label=\textup{(}\alph*\textup{)}]
\item\label{Toninibaschange Gammaa} Let $E_{\lW}\in\Crys(X/\lW)$ and
    set $E=(E_{\lW})_{|(X/\lS)_\crys}$. Assume that $S$ is smooth
    over $k$. Then there exists $N\colon \Z\arr \N$, depending
    only on $E_{\lW}$ and $g_0$, such that the $i$-th cohomology of the map (\ref{basechange affine}) is a $p^{N_i}$-isogeny isomorphism.
\item\label{Toninibasechange quasi-iso Gamma E lf} The map
    (\ref{basechange affine}) is an isomorphism if $E\in
    \Crys(X/\lS)$ is a flat crystal \textup{\cite[p.
    7.10]{BO78}}.
\item\label{Toninibasechange quasi-iso Gamma altro} The map
    (\ref{basechange affine}) is an isomorphism if $E \in
    \Crys(X/\lS)$ is $p$-torsion free and all $u_n\colon
    \lS'_n\arr \lS_n$  are either flat or the base change of a map to $\lW_n$.
\end{enumerate}
\end{thm}
Before proving this theorem, we consider two remarks.

\begin{rmk}\label{Wflatness after completion}
 If $M$ is a flat $W$-module, that is it is $p$-torsion free, then so is its $p$-adic completion. Indeed let $m_n\in M$ be a collection of elements such that $m_{n+1}-m_n=p^nx_n \in p^nM$ and $pm_n=p^ny_n \in p^nM$. Then $m_n=p^{n-1}y_n$ and
 \[
 p^ny_{n+1}-p^{n-1}y_n = p^n x_n \then y_n \in pM\then m_n\in p^n M.
 \]
\end{rmk}

\begin{rmk}\label{Replacing by ptorsion free}\cite[after Remark 2.5]{ES18}
 If $X$ is a smooth and quasi-compact $k$-scheme and $E_{\lW}\in \Crys(X/\lW)$, then there exists a $p$-torsion free $E_{\lW}'\in \Crys(X/\lW)$ and an isogeny $E_{\lW}\arr E_{\lW}'$.
 Indeed one can check locally, using Proposition \ref{Crystals abelian} and Theorem \ref{cry equivalent to quasi-nil}, that the sequence $E_{\lW}[p^n]$ stabilizes to a subobject $E_{\lW}[p^\infty]$ which is killed by a power of $p$. Thus $E_{\lW}'=E_{\lW}/(E_{\lW}[p^\infty])$ meet the requirements.
\end{rmk}

\begin{proof}[Proof of Theorem \ref{Toninibasechange Gamma}]

By Remark \ref{Replacing by ptorsion free} we can assume that $E$ is
$p$-torsion free in (a). If $E$ is a flat crystal, then $E$ is
$p$-torsion free by Lemma \ref{flat is torsion free}.

Now, for $n\in \N$, let $u_n = u\times \lW_n\colon \lS'_n\arr \lS_n$ and consider the base change map
\begin{equation}\label{partial basechange}
\dl u_n^* \rr\Gamma((X/\lS_n)_{\rm{crys}},E_n)\arr
\rr\Gamma((X'/\lS_n')_{\rm{crys}},h_{\rm{crys}}^*E_n).
\end{equation}
 Firstly we would like to prove that the $\rr\varprojlim$ of (\ref{partial basechange}) yields the map (\ref{basechange affine}).

If, for some $a\in \mathbb{N}$, $p^a=0$ in $\odi {\lS'}$, the map $u\colon \lS'\arr \lS$ factors through $u_n$ for $n\geq a$ and therefore
\[
\dl u^* \rr\Gamma((X/\lS)_{\rm{crys}},E) \simeq \dl u_n^* \rr\Gamma((X/\lS_n)_{\rm{crys}},E_n).
\]

So what remains is the case where $p$ is not nilpotent in
$\odi{\lS'}$ and $X/S$ is proper and $S=\Spec A/p$ (\emph{i.e.}
$I=(p)$). By Corollary \ref{RGamma Rlim} (c) we have that
$\rr\Gamma((X/\lS)_{\rm{crys}},E)$ is quasi-isomorphic to a
complex $P^\bullet$ of $A$-modules which is bounded above and it is made
by finite free $A$-modules. In this case, by Corollary \ref{RGamma Rlim} (b),
$\rr\Gamma((X/\lS_n)_{\rm{crys}},E_n)\simeq P_n^\bullet \coloneq P^\bullet\otimes_A A_n$ and
$$
\dl u_n^*P^\bullet_n \simeq P^\bullet_n \otimes_{A_n} A'_n.
$$
This is a complex of flasque projective systems in the sense of \cite[Remark B.1.4]{BO78}. In particular by \cite[Remark B.1.6]{BO78} we have
\[
\rr\varprojlim \dl u_n^*P^\bullet_n \simeq \varprojlim [(P^\bullet\otimes_A
A')\otimes_{A'} A'_n] \simeq P^\bullet\otimes_A A'.
\]
The last isomorphism holds because $P^\bullet\otimes A'$ is a
complex of finite free $A'$-modules which are therefore
complete. Since $P^\bullet\otimes_A A'\simeq \dl u^* \rr\Gamma((X/\lS)_{\rm{crys}},E)$ we get the result.
\begin{enumerate}[label=\textup{(}\alph*\textup{)}]
\item Applying Theorem  \ref{Toninibasechange} \ref{almostisoW} we
know that there exists $N\colon \Z\arr \N$, which depends only
on $E_{\lW}$ and $g_0$ (thus not on $n$), such that the $i$-th cohomology of (\ref{partial basechange}) is a $p^{N_i}$-isogeny.
Letting $n$ vary we can consider (\ref{partial basechange})
as a map of complexes in $D^-(\N,A'_*)$ whose $i$-th cohomology
is a $p^{N_i}$-isogeny. By Theorem \ref{Toninibasechange}
\ref{unifbounded} we can suppose $N_i=0$ for $i\gg 0$, and
by \cite[Remark B.1.6]{BO78} we have that $\rr^i\varprojlim = 0$
for $i\geq 2$. Now applying $\rr^i\varprojlim$ to \eqref{partial
basechange} we get our result by Lemma \ref{toninitoplemma}.
\item 
    We consider, as in 
    \ref{Toninibaschange Gammaa}, the map in (\ref{partial basechange}). 
Applying Theorem \ref{Toninibasechange} \ref{uflatEfree}  we get that the map
(\ref{partial basechange}) is a quasi-isomorphism. Again
applying $\rr^i\varprojlim$ to \eqref{partial basechange} yields
the quasi-isomorphism \eqref{basechange affine}.

\item The proof is exactly as in \ref{Toninibasechange quasi-iso Gamma E lf}, using Theorem \ref{Toninibasechange} \ref{EWnflatetc}.\qedhere
\end{enumerate}
\end{proof}
\begin{rmk}
 A result along the same lines is proven in \cite[Theorem 1.19]{Shi07I} and \cite[Proposition 5.3]{ES16}. 
\end{rmk}


\subsection{Pullback in the crystalline site
revisited}\label{pullback section}
Suppose that we are in \textbf{Siuation} \ref{base change convention}.
In what follows we collect some properties of pullback of
sheaves in the crystalline topoi, following the discussion in
\cite[Chapter III, Section 2.2, p. 196]{Ber74}. We denote by 
$$
h^{-1}_{\crys}\colon (X/\lS)_{\crys}^{\sim} \to (X'/\lS')_{\crys}^\sim
$$
the pullback  in the morphism of topoi
$h_{\crys}$ (not ringed topoi) induced by the morphism $h: X'\rightarrow X$. We instead denote by $h^*_{\crys}$ the pullback of $\odi
{X/\lS}$-modules.

\begin{defn}
     Given $\lT'\in (X'/\lS')_\crys$ and $\lT\in (X/\lS)_\crys$
     a $h$-PD-morphism $\lT'\to \lT$ is a PD-morphism $v\colon
     \lT'\to \lT$ which is compatible with $h$ and $u$. 
 
For $\lT'\in (X'/\lS')_\crys$ we define the category
\[
I^{\lT'}_h = \{ h\text{-PD-morphisms } \lT'\to \lT \text{ with }
\lT\in (X/\lS)_\crys \}.
\]
Given $x'\in \lT'$ we also define the category
\[
I^{x',\lT'}_h = \{ h\text{-PD-morphisms } \lV \to \lT \text{
with } \lT\in (X/\lS)_\crys \text{ and } x'\in \lV\subseteq \lT'
\text{ open} \}.
\]
\end{defn}

\begin{lem}\label{pullback of sheaves of sets}
 Let $F$ be a sheaf on $(X/\lS)_{\crys}$, $\lT'\in (X'/\lS')_\crys$ and $x'\in \lT'$. Then
 \begin{enumerate}
  \item $h^{-1}_{\crys}(F)$ is the sheafification of the
      presheaf
  $$
\lT'\mapsto \colim_{(q\colon \lT'\to \lT) \in I^{\lT'}_h}
F(\lT).
$$
\item for  $q\colon \lV \to \lT$ in $I^{x',\lT'}_h$ there is a canonical map
$$
q^{-1}(F_\lT) \to h^{-1}_\crys(F)_{\lV}.
$$
\item the set $I^{\lT',x'}_h$ is filtered; moreover taking stalks at $x'\in \lV\subseteq \lT$ of the maps in $(2)$ and passing to the limit we obtain an isomorphism
\[
\colim_{(q\colon \lV\to \lT) \in I^{x',\lT'}_h}
q^{-1}(F_\lT)_{x'} \to (h^{-1}_{\crys}(F)_{\lT'})_{x'}.
\]
 \end{enumerate}

\end{lem}
\begin{proof}
 Point $(1)$ is \cite[Chapter III, Section 2, eq
 (2.2.10)]{Ber74}, while $(2)$ is an easy consequence of
 $(1)$. The proof that $I^{x',\lT'}_h$ is filtered is given in
 the first paragraph of \cite[Chapter III, p. 199]{Ber74}.
 As in \cite[Chapter III, eq (2.2.11), p. 199]{Ber74}, taking a double limit in $(1)$ we have 
 \[
(h^{-1}_{\crys}(F)_{\lT'})_{x'} = \colim_{(q\colon \lV\to \lT)
\in I^{x',\lT'}_h} F(\lT).
\]
By definition of $I^{x',\lT'}_h$ it is easy to rewrite the above equation as
\[
(h^{-1}_{\crys}(F)_{\lT'})_{x'} = \colim_{(q\colon \lV\to \lT)
\in I^{x',\lT'}_h} (F_\lT)_{q(x')} = \colim_{(q\colon \lV\to
\lT) \in I^{x',\lT'}_h} q^{-1}(F_\lT)_{x'}.
\]
\end{proof}

\begin{lem}\label{pullback of sheaves of modules}
    Let $F$ be a sheaf of $\sO_{X/\lS}$-modules on $(X/\lS)_\crys$, $\lT'\in (X'/\lS')_\crys$ and $x'\in \lT'$. Then
 \begin{enumerate}
  \item for  $q\colon \lV \to \lT$ in $I^{x',\lT'}_h$ there is a canonical map
$$
q^*(F_\lT) \to h^*_\crys(F)_{\lV}.
$$
\item taking stalks at $x'\in \lV\subseteq \lT$ of the maps in $(2)$ and passing to the limit we obtain an isomorphism
\[
\colim_{(q\colon \lV\to \lT) \in I^{x',\lT'}_h} q^*(F_\lT)_{x'}
\to (h^*_{\crys}(F)_{\lT'})_{x'}.
\]
 \end{enumerate}
\end{lem}
\begin{proof}
 By definition we have
 \[
h^*_\crys(F)=h^{-1}_\crys(F)\otimes_{h^{-1}_\crys \odi{X/\lS}}
\odi{X'/\lS'}.
\]
Properties $(1)$ and $(2)$ follows from Lemma \ref{pullback of sheaves of sets}, taking into account that tensor products commute with filtered colimits.
\end{proof}

\begin{lem}\label{derived pullback of sheaves of modules}
 Let $F^\bullet$ be a complex of sheaves of $\sO_{X/\lS}$-modules on $(X/\lS)_\crys$, $\lT'\in (X'/\lS')_\crys$ and $x'\in \lT'$. Then
 \begin{enumerate}
  \item for  $q\colon \lV \to \lT$ in $I^{x',\lT'}_h$ there is a
      canonical map of complexes
  \[
      q^*(F_\lT^\bullet) \longrightarrow  h^*_\crys
      (F^\bullet)_{\lV};
 \]
 \item for all $j\geq 0$, taking $j$-th cohomology, stalks at $x'$ and passing to the limit we obtain an isomorphism 
 \[
 \colim_{(q\colon \lV\to \lT) \in I^{x',\lT'}_h} \Hl^j
 (q^*(F_\lT^\bullet))_{x'} \longrightarrow (\Hl^j (h^*_\crys
 F^\bullet)_{\lT'})_{x'};
 \]
 \item if $F^\bullet$ is bounded from above then we have a
     canonical isomorphism 
\[
 \colim_{(q\colon \lV\to \lT) \in I^{x',\lT'}_h} (\dl^j
 q^*(F_\lT^\bullet))_{x'} \longrightarrow ((\dl^j h^*_\crys
 F^\bullet)_{\lT'})_{x'}.
 \]
 \end{enumerate}
\end{lem}
\begin{proof} (1), (2) follows literally from Lemma \ref{pullback of
    sheaves of modules} $(1)$ and $(2)$ respectively. If
    $F^\bullet$ is bounded, then by \cite[p. 7.7-7.8]{BO78} we
    can replace $F^\bullet$ by a complex of flat
$\odi{X/\lS}$-modules. By \cite[Chapter III, Cor 3.5.2, p.
211]{Ber74} we have that $F_\lT$ is a complex of flat $\odi
\lT$-modules for any $\lT\in (X/\lS)_\crys$. We can therefore
replace $\dl^j$ in (3) by $\Hl^j$, but this is just (2). \end{proof}

\subsection{Crystalline base change}

\begin{defn}\label{higher direct image up to isogeny}
Let $f\colon X\rightarrow S$ be a morphism of $k$-schemes.
There is a morphism of ringed topoi \cite[\href{https://stacks.math.columbia.edu/tag/07IK}{Tag 07IK}]{stacks-project}
\[f_{\rm{crys}}\colon((X/\lW)^{\sim}_{\rm{crys}},
    \mathcal{O}_{X/\lW})\longrightarrow
    ((S/\lW)^{\sim}_{\rm{crys}}, \mathcal{O}_{S/\lW}).
\]
For a sheaf of $\sO_{X/\lW}$-modules $E$ on $(X/\lW)_\crys$  we consider the higher direct images
    $\rr^nf_{\rm{crys}*}E$ and also $K\otimes
    \rr^nf_{\rm{crys}*}E$, which belong to $K\otimes
    \Mod(\mathcal{O}_{S/\lW})$.
\end{defn}    
    
%

\begin{thm}\label{basechange for crystalpush}
 Consider a cartesian diagram
   \[
  \begin{tikzpicture}[xscale=1.5,yscale=-1.2]
    \node (A0_0) at (0, 0) {$X'$};
    \node (A0_1) at (1, 0) {$X$};
    \node (A1_0) at (0, 1) {$S'$};
    \node (A1_1) at (1, 1) {$S$};
    \path (A0_0) edge [->]node [auto] {$\scriptstyle{h}$} (A0_1);
    \path (A0_0) edge [->]node [left] {$\scriptstyle{f'}$} (A1_0);
    \path (A0_1) edge [->]node [auto] {$\scriptstyle{f}$} (A1_1);
    \path (A1_0) edge [->]node [auto] {$\scriptstyle{v}$} (A1_1);
  \end{tikzpicture}
  \]
of quasi-compact $k$-schemes with $f$ smooth and quasi-compact.
Let $E\in\Crys(X/\lW)$ and assume that $E$ is flat (resp. $S$ is smooth over $k$).
Then there is a natural map in $D((S'/\lW)_\crys)$
 \begin{equation}\label{basechange2}
          \dl v_{\crys}^*\rr f_{\crys_*}(E)\arr \rr f'_{\crys_*}h_{\crys}^*(E)
 \end{equation}
 which is an isomorphism (resp. induces isogeny on coholomogy).
\end{thm}

\begin{proof}

The definition of the map in the statement is also given in the
proof of \cite[Chapter V, Theorem 3.5.1, p. 342]{Ber74}.
Applying adjunction to the canonical map $\dl h^*_{\rm{crys}}(E)\rightarrow h^{*}_{\rm{crys}}(E)$
we obtain a map
$$E\rightarrow \rr h_{\rm{crys}*}( h^*_{\rm{crys}}(E)).$$
Applying $\rr f_{\rm{crys}*}$ we get
$$\rr f_{\rm{crys}*}(E)\rightarrow \rr  v_{\rm{crys}*}\rr
f'_{\rm{crys}*}( h^*_{\rm{crys}}(E)).$$ 
The map (\ref{basechange2}) is obtained applying adjunction
again, which is possible because $\rr f_{\rm{crys}*}(E)$ is
bounded: if $(U, \lT, \delta)\in(S/W)_{\rm{crys}}$ then
(\cite[\href{https://stacks.math.columbia.edu/tag/07MJ}{Tag
07MJ}]{stacks-project}, \cite[Corollaire V, 3.2.3, p. 328]{Ber74})
\[
(\rr f_{\rm{crys}*}(E))_\lT \simeq \rr (f_{|f^{-1}(U)})_{f^{-1}(U)/\lT*}(E_{|(f^{-1}(U)/\lT)_{\rm{crys}}}),
\] 
which is bounded uniformly thanks to Theorem  \ref{Toninibasechange} \ref{unifbounded}.

The case when $E$ is flat is essentially contained in \cite[Chapter V, Theorem
3.5.1, p. 342]{Ber74}, but we include the proof for
completeness. 

Let's fix $(U',\lT',\delta)\in (S'/\lW)_\crys$ and $x'\in \lT'$.  
 It is enough to check that the map
 \begin{equation}\label{stalk of cohomology base change}
(\dl^j v^*_\crys(\rr f_{\crys*} (E))_{\lT'})_{x'}\to (\rr^j f'_{\crys*}(h^*_\crys(E))_{\lT'})_{x'}
\end{equation}
is a quasi-isomorphism (resp. $p^{N_j}$-isogeny for some $N_j$) for all $\lT'$ and $x'$. We follow notation from \S \ref{pullback section}, for instance recall that $I^{x',\lT'}_v$ is the filtered category of $v$-PD-morphisms $\lV\to \lT$ where $x\in \lV \subseteq \lT'$ is an open and $\lT\in (S/\lW)_\crys$.
 
Let $q\colon \lV\to \lT$ be an object of $I^{x',\lT'}_v$. By Lemma \ref{derived pullback of sheaves of modules} we have maps
   \[
  \begin{tikzpicture}[xscale=5,yscale=-1.2]
    \node (A0_0) at (0, 0) {$\dl q^*(\rr f_{\crys*} (E)_\lT)$};
    \node (A0_2) at (2, 0) {$\rr f'_{\crys*}(h^*_\crys(E))_\lV$.};
    \node (A1_1) at (1, 1) {$\dl v^*_\crys(\rr f_{\crys*} (E))_\lV$};
    \path (A0_0) edge [->]node [auto] {$\scriptstyle{c_q}$} (A0_2);
    \path (A1_1) edge [->]node [below] {$\scriptstyle{b_q}$} (A0_2);
    \path (A0_0) edge [->]node [below] {$\scriptstyle{a_q}$} (A1_1);
  \end{tikzpicture}
  \]
By \cite[\href{https://stacks.math.columbia.edu/tag/07MJ}{Tag
07MJ}]{stacks-project} the map $c_q$ is the map
considered in Theorem \ref{Toninibasechange} \ref{EWnflatetc}
(resp. \ref{almostisoW}).
Therefore, $c_q$ is a quasi-isomorphism (resp. we find $N\colon \Z\to \N$ depending only on $f$ and
$E$, such that the map $\Hl^j(c_q)$ is an $p^{N_j}$-isogeny).

Now, on the diagram above, we take $j$-th cohomology and the stalk
at $x'$. The map $b_q$ becomes the map \eqref{stalk of
cohomology base change}. This map and, in particular, its source and
target do not depend on $q\in I^{x',\lT'}_v$. Let's call it
$B\to C$. Passing to the colimit for $q\in I^{x',\lT'}_v$ (at
the level of complexes) we get the diagram of the form
   \[
  \begin{tikzpicture}[xscale=4,yscale=-1.2]
    \node (A0_0) at (0, 0) {$\colim_q A_q$};
    \node (A0_2) at (2, 0) {$C$};
    \node (A1_1) at (1, 1) {$B$};
    \path (A0_0) edge [->]node [below] {$\scriptstyle{\alpha}$} (A1_1);
    \path (A1_1) edge [->]node [below] {$\scriptstyle{\gamma}$} (A0_2);
    \path (A0_0) edge [->]node [above, text centered]
        {$\scriptstyle{\colim_q\beta_q}$} (A0_2);
  \end{tikzpicture}
  \]
with the map $\alpha$ an isomorphism by Lemma \ref{derived
pullback of sheaves of modules}. If $c_q$ is a
quasi-isomorphism, then so is $\colim_q\beta_q$, hence so is $\gamma$. This finishes the proof in the case when $E$ is flat.

Let's now focus on the ``resp.'' case. Taking the limit of the exact sequence
\[
0 \to K_q \to A_q \xrightarrow{\beta_q} C \to D_q \to 0
\]
we obtain that
\[
\Ker(\gamma)\simeq \colim_q K_q\text{ and } \Coker(\gamma)\simeq
\colim_q D_q.
\]
Because all the $\beta_q$ are $p^{N_j}$-isogenies, $p^{N_j}$ kills all $K_q$, $D_q$ and therefore $\Ker(\gamma)$ and $\Coker(\gamma)$, as required.
\end{proof}

\section{Higher Push-forward of Isocrystals}\label{section:pushforward}

This section is dedicated to the proof of Theorem \ref{main theorem}.

\begin{thm}\label{pushforwardisanisocrystal}
 Let $f\colon X\rightarrow S=\Spec A$ be a smooth and proper
 morphism between smooth $k$-schemes, and let $\mathcal{A}$ be a
 $p$-adically complete flat lift of $A$ over $W$ and $E\in
 \mathrm{Crys}(X/\lW)$ be a $p$-torsion free crystal.  Then for each
 $n\in\N$
 there is a crystal $E_{X/\sA}^n$ in $\mathrm{Crys}(S/\lW)$ with
 a morphism of sheaves  $\eta_n\colon E_{X/\sA}^n\to
 \rr^nf_{\crys_*}(E)$ on the crystalline site $(S/\lW)_\crys$
 which induces the isomorphism 
 \[
 \varprojlim_e (E_{X/\sA}^n)_{\Spec(\sA/p^e)} \simeq
 \varprojlim_e (\rr^nf_{\crys_*}(E))_{\Spec(\sA/p^e)}.
 \]
Moreover, 
$$
\eta_n\otimes K\colon E_{X/\sA}^n\otimes K\to \rr^nf_{\crys_*}(E)\otimes K
$$
is an isomorphism and $E_{X/\sA}^n$ corresponds,  via
Theorem \ref{cry equivalent to quasi-nil}, to the $\sA$-module
$$\Hl^n((X/\lS)_{\rm{crys}}, E_{|(X/\lS)_{\rm{crys}}})$$ equipped
with a topologically quasi-nilpotent connection.
\end{thm}

\begin{proof}[Proof of Theorem \ref{main theorem} as a consequence of Theorem \ref{pushforwardisanisocrystal}]
  By Remark \ref{Replacing by ptorsion free} we can assume $\mathcal{E}=E\otimes K$, where $E\in \Crys(X/\lW)$ is $p$-torsion free.
    By Theorem \ref{pushforwardisanisocrystal} the statement is true when $S$
    is affine. By descent for isocrystals (\cite[Lemma
    0.7.5]{Ogus90}), we can conclude that an $\sO_{X/\lW}$-module on $(X/\lW)_\crys$ in the isogeny category
    is an isocrystal if and only if it is Zariski locally so. This finishes the proof.
\end{proof}

%
%

\begin{proof}[Proof of Theorem \ref{pushforwardisanisocrystal}]

 Set $\sA_e=\sA/p^e$, $\lS_e=\Spec \sA_e $, 
\[
E_\lS = E_{|(X/\lS)_{\rm{crys}}} \text{ and } H_n = \Hl^n((X/\lS)_{\rm{crys}},
E_{|(X/\lS)_{\rm{crys}}}).
\]

We construct the crystal $E_{X/\sA}^n$ in $\mathrm{Crys}(S/W)$ with the morphism $\eta_n\colon E_{X/\sA}^n\to \rr^nf_{\crys_*}(E)$.
%
%
Let $\lD(e)$ be the $p$-adic completion of the PD-envelope of $S$ inside $\lS\times_{\lW}\lS \dots\times_{\lW}  \lS$ (the fiber product over $\lW$ of $e$ copies of $\lS$). 
Since $S$ is smooth, the projections 
$$p_i:\lD(e) \arr \lS$$ 
are flat (\cite[3.32]{BO78}, \cite[\href{https://stacks.math.columbia.edu/tag/0912}{Tag 0912}]{stacks-project}). 
By Theorem \ref{Toninibasechange Gamma} \ref{Toninibasechange quasi-iso Gamma altro} we get canonical isomorphisms
$$ 
p_i^*\rr\Gamma((X/\lS)_{\rm{crys}},E_\lS)\arr \rr\Gamma((X/\lD(e))_{\rm{crys}},E_{|(X/\lD(e))_{\rm{crys}}}).
$$
Taking cohomology we also get canonical isomorphisms
\[
p_i^* H_n \arr \Hl^n((X/\lD(e))_{\rm{crys}},E_{|(X/\lD(e))_{\rm{crys}}}).
\]
This defines an HPD-stratification on the $\sA$-module $H_n$, which is finitely
generated by Corollary \ref{RGamma Rlim}. Similarly to
\cite[6.6]{BO78},  this HPD-stratification defines a crystal $E^n_{X/\sA}\in
\Crys(S/\lW)$. Let's recall here its construction.

For each object $\chi=(U, \lT, \delta)\in(S/\lW)_{\rm{crys}}$
with $\lT$ affine we get, thanks to \cite[\href{https://stacks.math.columbia.edu/tag/07K4}{Tag 07K4}]{stacks-project} and the smoothness of $S$, a commutative diagram 
$$\xymatrix{U\ar@{^{(}->}[r]\ar[d]&\lT\ar@{.>}[d]^-{\alpha_\chi}\\
S\ar[r] &\lS.}$$
We set $(E^n_{X/\sA})_\lT$ to be the quasi-coherent sheaf on $\lT$ associated
to $\alpha_\chi^*H_n$. The structure of the HPD-stratification 
allows us to define the transition morphisms
and to prove the functoriality of the correspondence $\chi=(U, \lT, \delta)\mapsto (E^n_{X/\sA})_\lT$.

By Theorem \ref{Toninibasechange Gamma} \ref{Toninibaschange
Gammaa} with $S'=U$ and $\lS'=\lT$ there exists $N\colon \Z\arr
\N$, depending only on $E$ and $f$, such that the $i$-th cohomology of
\begin{equation}\label{equation one}
 \gamma_\chi \colon \dl \alpha_\chi^* \rr\Gamma((X/\lS)_{\rm{crys}},E_\lS)\arr \rr\Gamma((f^{-1}(U)/\lT)_{\rm{crys}},E_{|(f^{-1}(U)/\lT)_{\rm{crys}}})
\end{equation}
is a $p^{N_i}$-isogeny.

Notice that
(\cite[\href{https://stacks.math.columbia.edu/tag/07MJ}{Tag
07MJ}]{stacks-project}, \cite[Corollaire V.3.2.3, p.
318]{Ber74})
\[
(\rr f_{\rm{crys}*}(E))_\lT \simeq \rr f_{f^{-1}(U)/\lT*}(E_{|(f^{-1}(U)/\lT)_{\rm{crys}}})
\]
is quasi-isomorphic to the complex of $\sO_\lT$-modules
associated to any complex of $\Hl^0(\sO_\lT)$-modules representing the right hand side of (\ref{equation one}).

Moreover there is a canonical map
\[
\iota_\chi\colon \alpha_\chi^*H_n\arr \Hl^n(\dl \alpha_\chi^*
\rr\Gamma((X/\lS)_{\rm{crys}},E_\lS)).
\]
Putting everything together we get a canonical morphism
\[
(\eta_n)_\lT\colon (E^n_{X/\sA})_\lT  \arr
(\rr^n f_{\rm{crys}*}(E))_\lT.
\]
If $\chi=(S, \lS_e, \delta_e)$ and $\alpha_\chi\colon \lS_e\arr
\lS$ is the obvious closed immersion, then, by Theorem \ref{Toninibasechange
Gamma} \ref{Toninibasechange quasi-iso Gamma altro}, the map $\gamma_\chi$ is a quasi-isomorphism and $(\eta_n)_{\lS_e}$ becomes
the map of quasi-
coherent sheaves on $\lS_e$ associated to the map
\[
H_n \otimes \sA_e \arr
\Hl^n((X/\lS_e)_{\rm{crys}},E_{|(X/\lS_e)_{\rm{crys}}}).
\]
By Corollary \ref{RGamma Rlim} the projective limit of the above maps is an
isomorphism as required. The limit $H_n$, which corresponds to $E^n_{X/\sA}$
via Theorem 1.25, is therefore the module with the topologically quasi-nilpotent connection in the statement.

It remains to show that $\eta_n\otimes K$ is an isomorphism. It is enough to show that there exists a $\overline N\in \N$ such that for all $\chi=(U, \lT, \delta)\in(S/W)_{\rm{crys}}$ the map $(\eta_n)_\lT$ is a $ p^{\overline N}$-isogeny. Since $\gamma_\chi$ is a $p^{N_n}$-isogeny, we have to prove the analogous statement for $\iota_\chi$.

Set $M: =\rr\Gamma((X/\lS)_{\rm{crys}},E_\lS)$. By \cite[Proposition 5.7.6,
with the convention on Dual Definition 5.2.3]{Wei94} there is a convergent spectral sequence
\[
E_2^{uv}=\dl^u\alpha_\chi^*(\Hl^v(M)) \Rightarrow
\dl^{u+v}\alpha_\chi^*(M)=H^{u+v}.
\]
Since $M$ is bounded there exists $l\in \N$ such that $E_2^{uv}=0$ for $v<0$ or
$v>l$. Moreover $E_2^{uv}=0$ if $u>0$. By Lemma \ref{almost degenerate spectral sequence} we obtain a map
\[
E_2^{0n}=\alpha_\chi^*(\Hl^n(M))\arr \dl^n \alpha_\chi^*(M)=H^n
\]
which coincides with the map $\iota_\chi$. 

Since $\Hl^v(M)=H_v$ is endowed with a topologically quasi-nilpotent connection
on $\sA$, by Lemma \ref{projective p inverted} there exists
$N_v\in \N$, depending only on $H_v$, such that
$\dl^q\alpha_\chi^*(\Hl^v(M))$ is killed by $p^{N_v}$ for any $q\neq
0$. Since $\dl^q\alpha_\chi^*(\Hl^v(M))=0$ for all $v<0$ or $v>l$, we can choose $N$ large, so that it kills
$\dl^q\alpha_\chi^*(\Hl^v(M))$ for all $q\neq 0$ and $v$.
Thus Lemma \ref{almost degenerate spectral sequence} tells us that $\iota_\chi$ is a $p^{N(l+1)}$-isogeny.
\end{proof}
\begin{rmk}\label{Xu comparison}
 We want to compare \cite[Theorem 1.9]{Xu2019} and Theorem \ref{main theorem}
 and, in particular, show how they are compatible. Assume the common settings
 for those results, that is, let $f\colon X\to S$ be a smooth and proper morphism of smooth $k$-schemes and $\E\in I_{\rm{conv}}(X/\lW)$, where $I_{\rm{conv}}(X/\lW)$ denotes the category of convergent isocrystals.
 
 By \cite[Theorem 0.7.2]{Ogus84}, there is a fully faithful functor $\iota\colon I_{\rm{conv}}(X/\lW) \to I_{\crys}(X/\lW)$ and similarly for $S$.
 Moreover, $\rr^if_{\rm{conv}*}(\E)\in I_{\rm{conv}}(S/\lW)$  by \cite[Theorem 1.9]{Xu2019}  and $\rr^if_{\crys *}(\iota (\E))\in I_{\crys}(S/\lW)$ by Theorem \ref{main theorem}.
 We claim that there is a canonical isomorphism
 \[
 \iota(\rr^if_{\rm{conv}*}(\E)) \simeq \rr^if_{\crys *}(\iota
 (\E))\text{ in } I_{\crys}(S/\lW).
 \]
 
By descent for isocrystals (\cite[Lemma 0.7.5]{Ogus90}) we can assume
that $S$ is affine and, by \ref{Replacing by ptorsion free}, choose a
$p$-torsion free crystal $E$ such that $\iota(\E) \simeq E\otimes K$. We use
the notations from Theorem \ref{pushforwardisanisocrystal} and freely refer to its proof. In particular we consider the schemes $\lD(e)$ with projections $p_i\colon \lD(e)\to \lS$ and the module $H_n$ with stratification defined at the beginning of the proof.

Since all $\lD(e) \to \lW$ are flat, the associated formal schemes $\lP(e)$ belong to the convergent site of $S/\lW$.
We use the description of $\iota\colon I_{\rm{conv}}(S/\lW) \to I_{\crys}(S/\lW)$ given in \cite[Section 3.20]{Xu2019}. 
Applying \cite[Theorem 3.22]{Xu2019} (or \cite[Theorem 2.36]{Shi07I}) to $X/\lP(e)$ (be aware that the
$g_{X/\lP(e),\crys*}$ in the reference is what we denoted by $g_{X/\lD(e)*}$)  we see that $H_n\otimes K$ is the module with stratification inducing $\iota(\rr^if_{\rm{conv}*}(\E))$ (see also the proof of \cite[Lemma 4.10]{Xu2019}).
This shows the claim.
\end{rmk}

\begin{proof}[Proof of Theorem \ref{underived base change} as a consequence of
    Theorem \ref{pushforwardisanisocrystal}]
 By Theorem \ref{basechange for crystalpush} there is an isogeny
 \[
 \Hl^n( \dl v_{\crys}^*\rr f_{\crys_*}(E)) \arr \rr^n f'_{\crys_*}h_{\crys}^*(E).
 \]
 Set $M\coloneqq \rr f_{\crys_*}(E)$. There is a canonical map
 \[
 \phi\colon v_\crys^*\rr^n f_{\crys*}(E) = v_\crys^*(\Hl^n(M))\arr \Hl^n(\dl v_\crys^*M).
 \]
 We have to prove that it is a $p^{N_n}$-isogeny with an
 $N_n\in \N$ depending only on $E$ and $f$. We are going
 to show that there exists $N_n\in \N$, depending only on $E$
 and $f$, such that for all $\chi=(U', \lT',\delta')\in
 (S'/\lW)_\crys$ the $\sO_{\lT'}$-linear map 
 \[
 \phi_\chi\colon  v_\crys^*(\Hl^n(M))_{\lT'}\arr \Hl^n(\dl v_\crys^*M)_{\lT'}
 \]
  is a
 $p^{N_n}$-isogeny.  To show this it is enough to show that for each
 $x'\in \lT'$ the map on stalks
 $$ \phi_{\chi,x'}\colon (v_\crys^*(\Hl^n(M))_{\lT'})_{x'}\arr
 \Hl^n((\dl v_\crys^*M)_{\lT'})_{x'}
$$ is a $p^{N_n}$-isogeny. Now we use notation
from \S \ref{pullback section}. Recall that $I^{x',\lT'}_v$ is
the filtered category of $v$-PD-morphisms $\lV\xrightarrow{u} \lT$ where $x\in \lV \subseteq \lT'$ is an open and $\lT\in (S/\lW)_\crys$.
Then we have a commutative diagram 
$$\xymatrix{
   u^*(\Hl^n(M)_\lT) \ar[r]\ar[d]& \Hl^n(\dl
    u^* M_\lT)\ar[d]\\ v_\crys^*(\Hl^n(M))_{\lV}\ar[r]& \Hl^n(\dl
    v_\crys^*M)_{\lV}.
}$$
 If we take the stalk at $x'$ in the above diagram, then the
 bottom horizontal map is exactly
 $\phi_{\chi,x'}$. Moreover, if we take the colimit
 of the vertical arrows over all $u\in I^{x',\lT'}_v$,
 then the vertical arrows are isomorphisms by Lemma \ref{derived pullback of sheaves
 of modules}.  
 Thus it is enough for us to show that the top horizontal arrow
 is a $p^{N_n}$-isogeny with $N_n$ depending only on $E$
 and $f$.

Now consider $(U', \lT',\delta')\in (S'/\lW)_\crys$,  $(S, \lT,
\delta)\in (S/\lW)_\crys$  and a commutative  diagram
   \[
  \begin{tikzpicture}[xscale=3,yscale=-1.2]
      \node (A0_0) at (0, 0) {${f'}^{-1}(U')$};
    \node (A0_1) at (1, 0) {$X$};
    \node (A1_0) at (0, 1) {$U'$};
    \node (A1_1) at (1, 1) {$S$};
    \node (A2_0) at (0, 2) {$\lT'$};
    \node (A2_1) at (1, 2) {$\lT$};
    \path (A0_0) edge [->]node [auto] {$\scriptstyle{h}$} (A0_1);
    \path (A2_0) edge [->]node [auto] {$\scriptstyle{u}$} (A2_1);
    \path (A1_0) edge [->]node [auto] {$\scriptstyle{v}$} (A1_1);
    \path (A1_0) edge [->]node [auto] {$\scriptstyle{}$} (A2_0);
    \path (A1_1) edge [->]node [auto] {$\scriptstyle{}$} (A2_1);
    \path (A0_0) edge [->]node [left] {$\scriptstyle{f'}$} (A1_0);
    \path (A0_1) edge [->]node [auto] {$\scriptstyle{f}$} (A1_1);
  \end{tikzpicture}
  \] where $v, u$ form a PD-map. We have to show that
 the map
  \[
  u^*\Hl^n(M_\lT)\arr \Hl^n(\dl u^*(M_\lT))
  \] is a $p^{N_n}$-isogeny for some $N_n\in\N$ depending only on
  $E$ and $f$.
  Notice that by Theorem \ref{Toninibasechange} \ref{unifbounded} the
  complex $M_\lT\simeq \rr
  f_{X/\lT*}(E_{|(X/\lT)_\crys})$ is  bounded  with a
  bound depending only on $f$.
  By \cite[Proposition 5.7.6, with the convention on Dual Definition 5.2.3]{Wei94} there is a convergent spectral sequence
\[
E_2^{ab}=\dl^au^*(\Hl^b(M_\lT)) \Rightarrow \dl^{a+b}u^*(M_\lT)=H^{a+b}.
\]
The upper bound of $M_\lT$ provides a number $l\in \N$ depending only
on $f$ (so independent of the choice of $\lT$), such that $E_2^{ab}=0$ for $b<0$ or $b>l$. Moreover $E_2^{ab}=0$ if $a>0$. By Lemma \ref{almost degenerate spectral sequence} we obtain a map
\[
E_2^{0n}=u^*(\Hl^n(M_\lT))\arr \dl^n u^*(M_\lT)=H^n
\]
which coincides with the map $\phi_\chi$. 

By Lemma \ref{almost degenerate spectral sequence} we must show that
there exists $N_n$, which depends only on $E$ and $f$, such that $\dl^a u^*(\Hl^b(M_\lT))$ is killed by $p^{N_n}$ for $a\neq 0$. We can assume that $\lT$ and $S$ are affine.
By Remark \ref{Replacing by ptorsion free} and Theorem \ref{pushforwardisanisocrystal} there exists a crystal $H\in
\Crys(S/\lW)$ which is isogenous to $\Hl^b(M)$. Thus it
is enough to look at $\dl^a u^*H_\lT$. By Theorem \ref{cry equivalent to
quasi-nil} $H$ corresponds to some $(P, \nabla)\in
\QNCf(S/\lW)$. Let $S_\lW=\Spec \sA\to \lW$ be a lift 
of $S$ as in Lemma \ref{affine lift} (2), so that $P$ is an $\sA$-module. The
smoothness of $S_{\lW_n}$ over $\lW_n$ for all $n\in\N$ and \cite[\href{https://stacks.math.columbia.edu/tag/07K4}{Tag 07K4}]{stacks-project} imply the existence of a map  $\lT\to S_\lW$
lifting the identity map of $S$ along $S\subseteq S_\lW$. In particular $P
\otimes \odi \lT \simeq H_\lT$. Applying Lemma \ref{projective p
inverted} to $P$ and  $F=\dl^au^*(-\otimes \odi \lT)$ we find
the $N_n\in \N$ depending only on $E$ and $f$ such that $p^{N_n}$ kills $F(P)=\dl^au^*H_\lT$ for $a\neq 0$.
\end{proof}

\section{The Künneth Formula}\label{section:kunneth}

In this last section we prove Theorems \ref{Kunneth} and
\ref{Abelian}.

\begin{proof}[Proof of Theorem \ref{Kunneth}]
Consider the following diagram 
\[
\xymatrix{1\ar[r]&\pi_1^{\rm{crys}}(Y/\lW,y)\ar[r]\ar@{=}[d]&\pi_1^{\rm{crys}}(X\times_kY/\lW,(x,y))\ar[r]\ar[d]&\pi_1^{\rm{crys}}(X/\lW,x)\ar[r]\ar@{=}[d]&1\\
1\ar[r]&\pi_1^{\rm{crys}}(Y/\lW,y)\ar[r]&\pi_1^{\rm{crys}}(X/\lW,x)\times_k\pi_1^{\rm{crys}}(Y/\lW,y)\ar[r]&\pi_1^{\rm{crys}}(X/\lW,x)\ar[r]&1.}
\]
It is enough to show that the top sequence is exact.
Consider the diagram
\begin{equation}\label{kunneth diagram}
    \xymatrix{
        & Y\ar[r]^-x\ar[d]_-g&X\times_kY\ar[d]^{p_1}\\
                                                 &\Spec k\ar[r]^-u&X.}
    \end{equation}
 Since $x$ is a section of the projection, it gives a closed embedding on fundamental group schemes, while the projection yields a surjection on fundamental group schemes.
 We are going to apply \cite[Theorem A.1 (iii)]{EPS07} to prove
 the exactness in the middle. So we have to check:
 \begin{enumerate}[label=(\alph*)]
     \item If $\sE\in I_{\crys}(X\times_kY/\lW)$, then $x_\crys^*\sE$
         is a trivial object in $I_{\crys}(Y/\lW)$ if and only if there exists
         $\sF\in I_\crys(X/\lW)$ such that $p_{1\crys}^*\sF\simeq \sE$.
     \item We have to check that for any isocrystal $\sE\in I_{\rm{crys}}(X\times_kY/\lW)$, the maximal trivial 
subobject of $x_\crys^*\sE$ comes from a subobject $\sF\subseteq
\sE$, where $\sF$ is defined over $X/\lW$. 
    \item If $\stG\in I_\crys(Y/\lW)$, then there exists $\sE\in
        I_\crys(X\times_kY/\lW)$ such that $\stG$ is a
        subobject of $x_\crys^*\sE$. 
 \end{enumerate}
  Condition (c) follows because $x$ is a section of the
  projection $X\times_kY\xrightarrow{p_2}Y$. Also the "if" part
  of (a) is obvious from \eqref{kunneth diagram}, and the "only
  if" part is a consequence of (b). Thus let's focus on (b).

  Since $p_{1\crys_*}$ and $p_{1\crys}^*$ are a pair of adjoint
  functors between the category of sheaves of $\sO$-modules on
  $(X\times_kY/\lW)_\crys$ and that on $(X/\lW)_\crys$, and thanks to Theorem \ref{main theorem}, the
  induced pair of functors between the isogeny categories  $I_\crys(X\times Y/\lW)$ and
  $I_\crys(X/\lW)$ are also adjoint to each other.   The map
  $p_1$ induces a map on fundamental group schemes
  $$\pi_1^\crys(p_1)\colon
  \pi_1^{\rm{crys}}(X\times_kY/\lW,(x,y))\arr
  \pi_1^{\rm{crys}}(X/\lW,x)$$which is surjective because
  $p_{1}$ has a section. It follows that $p_{1\crys_*}$ on isocrystals corresponds to taking invariants by the kernel of $\pi_1^\crys(p_1)$. In particular the map
\[
\sF\coloneq p_{1\crys}^{*}p_{1\crys_*}\sE\arr \sE
\]
is injective. 

The same argument applied to $I_\crys(Y/\lW)$ and
  $I_\crys(\Spec k/\lW)$ shows that  \[
 g_{\crys}^{*}g_{\crys_*}x_\crys^*\sE\arr x_\crys^*\sE
\]
is injective and  $g_{\crys}^{*}g_{\crys_*}x_\crys^*\sE$ is
the maximal trivial subobject of $x_\crys^*\sE$.

Using the base change isomorphism in Theorem \ref{underived
base change} in \eqref{kunneth diagram}, we can conclude that
applying  $x_\crys^{*}$ to $\sF\arr\sE$ we get the map $g_{\crys}^{*}g_{\crys_*}x_\crys^*\sE\arr x_\crys^*\sE
$ as required.
\end{proof}

\begin{proof}[Proof of Theorem \ref{Abelian}] By the binary
    operation on $\pi_1^{\rm{crys}}(A/\lW, 0)$ induced by the
    addition of the abelian variety $A$,
    $\pi_1^{\rm{crys}}(A/\lW, 0)$ becomes group object in the
    category of affine group schemes over $K$. Then, by the
    calculation given in \cite[Theorem 5.4.2]{EckHil62},
    $\pi_1^{\rm{crys}}(A/\lW, 0)$ is an abelian group scheme.
    \end{proof}

\printbibliography

\end{document}